\pdfoutput=1

\documentclass[onefignum,onetabnum, final]{siamart171218}

\usepackage{lipsum}
\usepackage{amsfonts}
\usepackage{graphicx}
\usepackage{epstopdf}
\usepackage{algorithmic}
\usepackage{multirow}
\usepackage{multicol}
\usepackage{caption}
\usepackage{subcaption}

\usepackage{setspace}
\newcommand{\revision}[1]{\textcolor{black}{#1}}

\ifpdf
  \DeclareGraphicsExtensions{.eps,.pdf,.png,.jpg}
\else
  \DeclareGraphicsExtensions{.eps}
\fi

\renewcommand{\baselinestretch}{1.5} 
\renewcommand{\epsilon}{\varepsilon}

\renewcommand{\phi}{\varphi}
\renewcommand{\H}{\mathcal{H}}

\newcommand{\Z}{\mathbb{Z}}
\newcommand{\R}{\mathbb{R}}
\newcommand{\C}{\mathbb{C}}

\newcommand{\N}{\mathbb{N}}

\newmuskip\pFqmuskip
\newcommand*\pFq[6][8]{%
  \begingroup 
  \pFqmuskip=#1mu\relax
  \mathcode`\,=\string"8000
  \begingroup\lccode`\~=`\,
  \lowercase{\endgroup\let~}\pFqcomma
  {}_{#2}F_{#3}{\left[\genfrac..{0pt}{}{#4}{#5};#6\right]}%
  \endgroup
}
\newcommand{\pFqcomma}{\mskip\pFqmuskip}

\allowdisplaybreaks

\newsiamremark{remark}{Remark}
\newsiamremark{hypothesis}{Hypothesis}
\crefname{hypothesis}{Hypothesis}{Hypotheses}
\newsiamthm{claim}{Claim}

\headers{Super-time-stepping methods}{Zheng Tan, Tariq D. Aslam, and Andrea L. Bertozzi}

\title{Explicit monotone stable super-time-stepping methods for finite time singularities}

\author{Zheng Tan\thanks{Department of Mathematics, UCLA, Los Angeles, CA 90095 (\email{zhengtan@math.ucla.edu}, \email{bertozzi@ucla.edu}.)}
\and Tariq D. Aslam\thanks{Los Alamos National Laboratory, Los Alamos, New Mexico 87545
  (\email{aslam@lanl.gov}).
  TDA was supported by the Advanced Simulation and Computing Program (ASC) at Los Alamos National Laboratory. Los Alamos National Laboratory is operated by Triad National Security, LLC, for the National Nuclear Security Administration of U.S Department of Energy (Contract No. 89233218CNA000001).}
\and Andrea L. Bertozzi*\thanks{ALB and ZT are supported by ONR grant DOD N00014-23-1-2565 and NSF grant DMS-2152717.}}

\usepackage{amsopn}

\makeatletter
\newcommand*{\addFileDependency}[1]{
  \typeout{(#1)}
  \@addtofilelist{#1}
  \IfFileExists{#1}{}{\typeout{No file #1.}}
}
\makeatother

\newcommand*{\myexternaldocument}[1]{%
    \externaldocument{#1}%
    \addFileDependency{#1.tex}%
    \addFileDependency{#1.aux}%
}

\ifpdf
\hypersetup{
  pdftitle={Stabilized Runge-Kutta Method for finite-time Singularity},
  pdfauthor={Zheng Tan, Tariq D. Aslam, and Andrea L. Bertozzi}
}
\fi

\myexternaldocument{ex_supplement}

\begin{document}

\renewcommand{\baselinestretch}{1}\normalsize  

\maketitle

\begin{abstract}
We explore a novel way to numerically resolve the scaling behavior of finite-time singularities in solutions of nonlinear parabolic PDEs.  The Runge--Kutta--Legendre (RKL) and Runge--Kutta--Gegenbauer (RKG) super-time-stepping methods were originally developed for nonlinear complex physics problems with diffusion.  These are multi-stage single step second-order, forward-in-time methods with no implicit solves.  The advantage is that the timestep size for stability scales with stage number $s$ as $\mathcal{O}(s^2)$. Many interesting nonlinear PDEs have finite-time singularities, and the presence of diffusion often limits one to using implicit or semi-implicit timestep methods for stability constraints.  Finite-time singularities are particularly challenging due to the large range of scales that one desires to resolve, often with adaptive spatial grids and adaptive timesteps.  
Here we show two examples of nonlinear PDEs for which the self-similar singularity structure has time and space scales that are resolvable using the RKL and RKG methods, without forcing even smaller timesteps. Compared to commonly-used implicit numerical methods, we achieve significantly smaller run time while maintaining comparable accuracy. 
We also prove numerical monotonicity for both the RKL and RKG methods under their linear stability conditions for the constant coefficient heat equation, in the case of infinite domain and periodic boundary condition, leading to a theoretical guarantee of the superiority of the RKL and RKG methods over traditional super-time-stepping methods, such as the Runge-Kutta-Chebyshev (RKC) and the orthogonal Runge-Kutta-Chebyshev (ROCK) methods. Code can be found at \url{https://github.com/ZT220501/SRK-Singularity}.
\end{abstract}



\section{Introduction}
\label{sec:intro}
Non-linear partial differential equations (PDE) describe many types of physical systems, yet most of these models are stiff and difficult to solve using explicit numerical methods. \revision{Those with finite-time singularities are drawing increasing attention, most notably, the exploration of the unsolved Clay prize problem regarding finite-time singularity in the Navier-Stokes equation~\cite{Hou2022, Moffatt_Kimura_2023, RUKAVISHNIKOV2023115218, doi:10.1137/140966411}, and the recent development of neural networks~\cite{PhysRevLett.130.244002}}. In order to guarantee the validity and robustness of the exploration of finite-time singularity, advanced numerical methods with rigorous theoretical guarantees are needed. 

In this paper, we introduce fully explicit super-time-stepping methods to resolve finite-time singularities in nonlinear parabolic PDE. Ideal candidates for such methods are problems for which the singularity has space-time scaling that is consistent with the time-step limitations of the numerical method.  finite-time singularities are often characterized by self-similar behavior in which the smallest spatial scale scales like a power law in time.  In the range of such singularities, one would need the timestep to be constrained to scale like the spatial grid to this inverse power in order to accurately track the singularity.  By using super-time-stepping methods, we are able to resolve the order-one dynamics leading into the singularity along with the power-law scaling of the singularity itself, using adaptive mesh refinement.  We compare such methods with more commonly-used implicit methods that do not have timestep constraints that scale with the grid size.  

Two well-known problems that possess such scaling are the semilinear heat equation~\cite{bergerkohn}
\begin{equation}
\label{eqn:semiheat}
\begin{cases}
    u_t=u_{xx}+u^p,\quad p>1\quad (x, t)\in[-a, a]\times(0, T],\\
    u(a, t)=u(-a, t)=0,\\
    u(x, 0)=u_0(x),
\end{cases}
\end{equation}
and the axisymmetric surface diffusion equation~\cite{mullins1957thermal}
\begin{subequations}
    \label{eqn:surfacediffusion}
    \begin{equation}
    \begin{cases}
    \frac{\partial r}{\partial t}=\frac{1}{r}\frac{\partial}{\partial z}\Big[\frac{r}{\sqrt{1+r_z^2}}\frac{\partial}{\partial z}\mathcal{H}\Big]\quad (z, t)\in[-a, a]\times(0, T],\\
    r(a, t)=r(-a, t),\\
    r(z, 0)=r_0(z),
    \end{cases}
    \end{equation}
\text{where}
    \begin{equation}
    \mathcal{H}=\frac{1}{r\sqrt{1+r_z^2}}-\frac{r_{zz}}{(1+r_z^2)^{3/2}}.
    \end{equation}
\end{subequations}

Both \ref{eqn:semiheat} and \ref{eqn:surfacediffusion} are stiff equations with finite-time singularities, a common occurrence in problems in nonlinear physics~\cite{Eggers_Fontelos_2015}. With suitable initial conditions, the semilinear heat equation blows up to infinity in finite-time, whereas the surface diffusion equation dynamics pinches off to zero in finite-time (with a discontinuity forming in the first derivative). Theoretical analysis for these equations has been done extensively in the literature, which will be discussed in more details in Section~\ref{sec:pdeproperty}.

\revision{Many numerical approaches have been designed to resolve the  theoretical scaling analysis of finite time singularities.  These include both adaptive mesh refinement with (semi-)implicit schemes analysis~\cite{bernoff1998axisymmetric, COLEMAN1995123, COLEMAN1996,Symmetric,BBDK} and schemes based on dynamic rescaling in similarity variables~\cite{bergerkohn, LANDMAN1991393, LEMESURIER198878, Papanicolaou1994,hou2024l2basedstabilityblowuplog}. Semi-implicit methods may be slower than explicit methods since large systems of implicit equations need to be solved for each update}. Whereas explicit methods can save the computational cost of solving large systems of equations, they often have timestep restrictions for stability.  

\revision{Chebyshev, Legendre, and Gegenbauer polynomials are chosen in super-time-stepping Runge--Kutta schemes (RKC, RKL, RKG, respectively) based on their stability and spectral properties (see discussion in Sec.~\ref{subsec:rkl_rkg}). Chebyshev polynomials are ideal for stiff diffusive problems due to their optimal stability over real negative eigenvalues, enabling large explicit timesteps~\cite{hundsdorfer2013numerical}. But RKC schemes, while linearly stable, do not satisfy the convex monotone property~\cite{RKL2}. RKL and RKG both satisfy this monotone property as proven here and RKC and RKG are the focus of the current investigation.} To address the issues above, we use two classes of super-time-stepping methods, the Runge--Kutta--Legendre (RKL) methods~\cite{RKL2} and the Runge--Kutta--Gegenbauer (RKG) methods~\cite{SKARAS_RKG2021109879, OSULLIVAN_RKG2019209}, to solve for stiff PDEs~\ref{eqn:semiheat} and~\ref{eqn:surfacediffusion}, along with other techniques such as Strang splitting~\cite{strangsplitting}. Adaptive mesh refinement is also deployed so that as the solution approaches blow-up time, the spatial grid size, and therefore the timestep becomes correspondingly smaller. The RKL and RKG methods are $s$-stage Runge-Kutta type explicit methods, where $s\in\N$ is a number that can be freely chosen. Both methods can be of first or second-order based on construction, and their orders are independent of the choice of $s$. We call the first-order and second-order RKL methods as RKL1 and RKL2; similarly we call the first-order and second-order RKG methods as RKG1 and RKG2. 

The main effect of the stage number $s$ is in the restriction of the timestep size with respect to the spatial grid size. When $s$ increases, the stability restriction relaxes quadratically with respect to the growth of $s$. Based on this property of the RKL and RKG methods, we adaptively choose $s$ so that when the solution has not yet blown up, a large timestep can be used to quickly drive the solution into the blow-up regime without loss of stability and accuracy; when the solution is near the blow-up time, since small $dt$ is required to resolve the PDE scaling, we gradually decrease $s$ so that less run time is achieved.

The second-order parabolic equations and plenty of physical equations have the maximum principle. To capture such a property, we need to ensure the numerical monotone stability of the numerical method. The primary goal for designing the RKL methods was to ensure the monotone stability when solving the second-order parabolic equation for the entire numerical stability region, while the RKG methods were designed to further ensure the monotone stability for the Dirichlet boundary condition. Previous studies~\cite{Meyer2012, RKL2} have numerically verified the monotonicity of the numerical method up to a size sufficiently large $s$; however, to the best of our knowledge, no theoretical proof has been established in the existing literature. In this paper, by transforming the monotonicity condition into an equivalent combinatorial summation problem and estimating the oscillation in the summation, a rigorous proof for RKL methods is proposed when applied to the constant coefficient heat equation under the numerical stability constraint, without the boundary condition or for the periodic boundary condition. The monotone stability for RKG methods are also established based on the proof of RKL monotonicity, along with inductive property of the hypergeometric function.

This paper is organized as follows. In Section~\ref{sec:pdeproperty} we introduce the basic properties and the scalings we need to resolve in both of the PDEs. In Section~\ref{sec:alg} we revisit the RKL and RKG methods in detail and formulate our adaptive application of the RKL and RKG methods to equation~\ref{eqn:semiheat} and~\ref{eqn:surfacediffusion}, including the Runge-Kutta formulations of the methods, the choice of substep number $s$, and the adaptive mesh refinement strategy. Section~\ref{sec:experiments} describes the details of our experiments to test the run time, accuracy, and robustness of the numerical method based on the existing theoretical asymptotic estimation and scalings. Section~\ref{sec:monotoneproof} formulates the monotonicity results and gives a detailed proof of them. Finally, we conclude in Section~\ref{sec:conclusions} and provide directions for future numerical and theoretical works.

\section{Properties of the PDEs}
\label{sec:pdeproperty}
In this section, we introduce the essential properties for solutions of the PDEs~\ref{eqn:semiheat} and~\ref{eqn:surfacediffusion}. In Section~\ref{sec:experiments}, the properties are tested numerically in order to verify the validity of our numerical simulation.

\subsection{Semilinear Heat Equation}
\label{semiheat_property}
The semilinear heat equation~\ref{eqn:semiheat} is a well-known PDE with finite-time singularity~\cite{bergerkohn, BricmontKupiainen, MerleZaag}; more precisely, with non-negative initial condition $u_0(x)$, there exists some $T>0$ such that $\underset{t\to T}{\lim}\|u(x, t)\|_\infty=\infty$. The asymptotic behavior of the semilinear heat equation has been studied extensively; a generic stable self-similar formulation was first proposed by M. Berger and R. Kohn~\cite{bergerkohn} to be
\begin{equation}
\label{eqn:semiheat_asymp}
    u(x, t)\sim(T-t)^{\frac{-1}{p-1}}\Big[(p-1)+\frac{(p-1)^2}{4p}\frac{x^2}{(T-t)|\log(T-t)|}\Big]^{\frac{-1}{p-1}}.
\end{equation}
Following that, it was proved by J. Bricmont and A. Kupiainen~\cite{BricmontKupiainen} via analysis for the unstable manifolds of the long-time dynamical system formulation of the problem. Later, F. Merle and H. Zaag~\cite{MerleZaag} proved the same result using a similar but simpler method as in~\cite{BricmontKupiainen}. Recently, a novel approach based on the $L^2$ stability analysis was used to prove the same asymptotic behavior for the case where $p=2$~\cite{hou2024l2basedstabilityblowuplog}.

For simplicity, we follow the setup in~\cite{bergerkohn} to assume that our initial condition $u_0(x)$ is symmetric with respect to the origin \revision{and has its maximum there}, so that the singularity is always at $x=0$. With the assumption above and based on the asymptotic behavior~\ref{eqn:semiheat_asymp}, one observation by plugging in $x=0$ is that $\|u(\cdot, t)\|_{L^\infty}\sim\Big[(T-t)(p-1)\Big]^{-\frac{1}{p-1}}$; taking logarithm on both sides, we get
\begin{equation}
\label{eqn:semiheat_max}
\log\|u(\cdot, t)\|_{L^\infty}=-\frac{1}{p-1}\log(T-t)-\frac{1}{p-1}\log(p-1)
\end{equation}
which shows that the maximum value of $u(x, t)$ is linear with respect to the time remaining before the blow-up in the $\log$-$\log$ scale, with slope $-\frac{1}{p-1}$. 

In order to perform an accurate numerical simulation, it is essential to understand and resolve the intrinsic scaling in the solution of the PDE. 
The numerical simulation near the blow-up time must resolve the scaling $dt\sim dx^2$ entailed by the PDE itself with logarithmic correction; this indicates that we should use small number of substages $s$ in the stabilized Runge-Kutta methods near the blow-up time.   In order to resolve the structure of the singularity we need to choose $dt\sim \mathcal{O}(dx^2)$, 
which allows us to use a smaller $s$ near the blow-up, resulting in more computational efficiency. 

Finally, we could also estimate the speed with which the maximum of $u(x, t)$ increases. Taking the derivative of~\ref{eqn:semiheat_asymp} on both sides, $u_t(0, t)\sim[(T-t)(p-1)]^{-\frac{p}{p-1}}$; taking the logarithm, it's easy to derive that
\begin{equation}
\label{eqn:semiheat_timederivative}
\log u_t(0, t)\sim p\log\|u(\cdot, t)\|_{\infty},
\end{equation}
so that the growth rate of the maximum value is linear in the maximum value itself in the $\log$-$\log$ scale with slope $p$, without other additional parameters.

\subsection{Surface Diffusion Equation}
The surface diffusion equation 
\begin{equation}
\label{eqn:surfdiff}
U=\nabla_s^2\mathcal{H}
\end{equation}
describes the motion of a surface driven by its mean curvature; here $U$ denotes the surface normal velocity, $\nabla_s^2$ denotes the surface Laplacian, and $\mathcal{H}$ denotes the normalized mean curvature of the evolving surface. It was first proposed and analyzed by Mullins~\cite{mullins1957thermal, mullins1995mass, Mullins1999} to model the evolution of microstructure in polycrystalline materials. Other factors neglected in the original model were analyzed, including the anisotropy case~\cite{CAHNTYLOR19941045, davi1990motion, klinger2001effects, XIN20032305}, the multiple grooves case~\cite{martinpaul2007}, and the volume diffusion case~\cite{HARDY1991467}. Special simplified cases of were also extensively studied in the existing literature, such as the axisymmetric case~\cite{WONG199855, bernoff1998axisymmetric} and the linearized case~\cite{mullins1957thermal, martinpaul2007, kalantarova2021self}. Self-similar solutions were derived in each of the two cases above. 

Besides theoretical analysis, numerical computations have also been done for the surface diffusion equation in various cases. Coleman, Falk, and Moakher~\cite{COLEMAN1995123, COLEMAN1996} first computed the axisymmetric surface diffusion equation using a finite element discretization and indicates that with a symmetric initial condition, there exists a pinchoff in finite-time. Following that, Bernoff, Bertozzi, and Witelski~\cite{bernoff1998axisymmetric} used backward Euler method along with adaptive mesh refinement strategy to compute the pinchoff with high resolution. Alternating direction implicit (ADI) methods are also used to solve the fluid flows driven by surface diffusion in higher dimensions~\cite{WITELSKI2003331}. However, all the existing methods are fully implicit or semi-implicit since the stiffness constraint for a usual explicit method is overly restrictive; therefore, the computational cost for solving the nonlinear system of equations becomes expensive. Therefore, the fully explicit RKL and RKG methods would largely reduce the run time.

In this paper, we focus on the three dimensional axisymmetric case. By plugging into the polar coordinate using $r\equiv r(\theta, z, t)$ where $r$ is considered to be the graph of a function, Equation~\ref{eqn:surfdiff} becomes
$$\frac{r}{q}\frac{\partial r}{\partial t}=\nabla_s^2\mathcal{H},\quad q=\sqrt{r^2(1+r_z^2)+r_\theta^2}$$
\revision{where $\nabla_s^2$ denotes the polar surface Laplacian
$$\nabla^2_s=\frac{1}{q}\Big[\frac{\partial}{\partial z}\Big(\frac{r^2+r_\theta^2}{q}\frac{\partial}{\partial z}-\frac{r_zr_\theta}{q}\frac{\partial}{\partial\theta}\Big)+\frac{\partial}{\partial\theta}\Big(\frac{1+r_z^2}{q}\frac{\partial}{\partial\theta}-\frac{r_zr_\theta}{q}\frac{\partial}{\partial z}\Big)\Big].$$}
By assuming the axisymmetric case, $r\equiv r(z, t)$, we get exactly Equation~\ref{eqn:surfacediffusion}. The similarity solution derived in~\cite{bernoff1998axisymmetric} takes the form of
\begin{equation}
\label{eqn:surfdiffscaling}
r(z, t)\sim\tau^{1/4}\overline{R}_i\Big(\frac{z}{\tau^{1/4}}\Big),\quad\tau=T-t.
\end{equation}

In contrast to the blow-up behavior of the semilinear heat equation, the similarity solution~\ref{eqn:surfdiffscaling} has a pinchoff at $z=0$ in finite-time, $\underset{t\to T}{\lim}\min r(z, t)=0$, with non-negative initial condition symmetric with respect to the origin. Therefore, a natural characteristic scaling within the PDE itself is $z^4\sim(T-t)$, and any numerical simulation needs to resolve the correct scaling $dt\sim dx^4$ near the pinchoff time to obtain the correct solution. Another result in~\cite{bernoff1998axisymmetric} shows that there exists a countable family of self-similar symmetric pinchoff solutions $\overline{R}_i(\eta)$, $i\in\N_0$, each with a different far-field cone angle $\overline{R}_i(\eta)\sim\pm c_i\eta$ as $\eta\to\pm\infty$, such that $c_i>c_{i+1}$ for all $i\in\N_0$. \revision{The half-cone angles of the self-similar solutions are given by $\phi_i=\arctan(c_i)$, where the physically relevant solution is the one with $i=0$.} Therefore, we expect our numerical simulation to capture the similarity solution with the correct cone angle. Finally, we also want to simulate the rate of change of the similarity solution. By a calculation in~\cite{bernoff1998axisymmetric}, we know that $\overline{R}_0(0)^4/4=0.060575684$, therefore $\Big|\frac{d}{dt}\overline{R}_0(0)\Big|=0.060575684/\overline{R}_0(0)^3$; \revision{taking the $\log$-$\log$ scale, we clearly observe that the rate of change of the minimum value scales like the negative of the minimum value cubed}.

In order to capture the fine scaling in both of the PDEs above, an important strategy we use is the adaptive mesh refinement in the simulations. We gradually add grid points near the singularity at the origin along with the growth or pinchoff of the solution, so that the solution computed near the singularity is computed on a fine mesh while away from the blowing up point is computed on a coarse mesh. Fig.~\ref{fig:meshrefinement} shows the mesh refinement for two times. \revision{The criterion for the time to perform a mesh refinement is based an easily measurable criteria for the evolving solution of the PDE}, and more details can be found in Appendix~\ref{app:amr}.

In Section~\ref{sec:alg}, we describe details of the numerical methods we used. Following that, Section~\ref{sec:experiments} describes various experiments based on the theoretical behaviors described in the current section in order to verify the validity of our numerical method. Different initial number of grid points are used as a proof of the robustness of our numerical method. The run time comparison between the explicit super-time-stepping methods and the classical (semi-)implicit methods is also listed in Section~\ref{sec:experiments} so that the superiority of our method is transparent.

\begin{figure}[!t]
    \centering
    \includegraphics[width=0.32\linewidth]{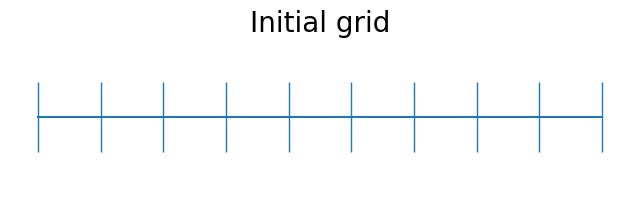}
    \includegraphics[width=0.32\linewidth]{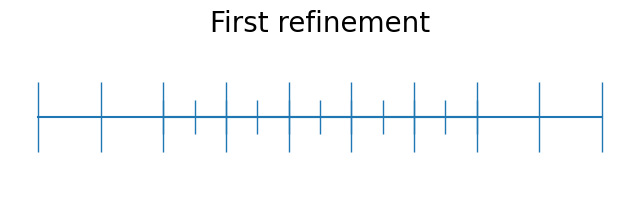}
    \includegraphics[width=0.32\linewidth]{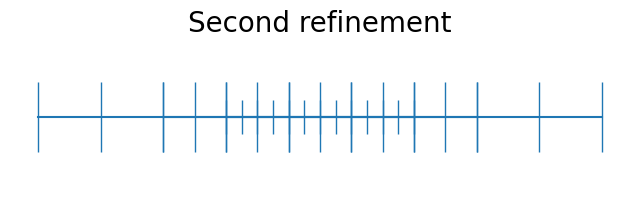}
    \caption{Adaptive mesh refinement strategy \revision{similar to the one used in \cite{Symmetric}}. The above three plots show the process of two mesh refinements around the origin. Each vertical $|$ represents a grid point, and the shorter $|$ at each level represents the new grid points added.}
    \label{fig:meshrefinement}
\end{figure}

\section{Numerical method}
\label{sec:alg}
This section gives a brief introduction to the numerical method we use, including the explicit super-time-stepping methods and the adaptive mesh refinement strategy based on the property of the PDEs themselves. We denote the initial timestep as $dt$, the spatial grid size as $dx$, and the number of stages for both the RKL and RKG methods as $s$. The central difference spatial approximation is used for solving both PDEs, while Strang splitting is also used for the semilinear heat equation in order to adopt the RKL and RKG methods for the parabolic portion of the right-hand-side operator.

\subsection{Runge--Kutta--Legendre and Runge--Kutta--Gegenbauer Methods}
\label{subsec:rkl_rkg}
The Runge--Kutta--Legendre methods~\cite{RKL2} and Runge--Kutta--Gegenbauer methods~\cite{OSULLIVAN_RKG2019209, SKARAS_RKG2021109879} are super-time-stepping methods to solve for an ordinary differential equation (ODE) system 
\begin{equation}
\label{autoode}
    \frac{du}{dt}=\mathbf{M}u
\end{equation}
where $\mathbf{M}$ is a symmetric, constant coefficient matrix that represents the discretization of a parabolic operator; the operator $\mathbf{M}$ is required to have non-positive, real eigenvalues to ensure convergence. 

The stability and the derivation of the RKL and RKG methods are based on the shifted Legendre polynomials and Gegenbauer polynomials, respectively. The general form of the RKL and RKG methods is
\begin{align}
u^{t+dt}=R_s(dt\mathbf{M})u^t
\end{align}
where $R_s(dtM)$ is the ``stability polynomial" of the numerical method. Numerical stability is ensured if $|R_s(dt\lambda)|\le1$ for all eigenvalues $\lambda$ of $M$; therefore, boundedness of Legendre polynomial and Gegenbauer polynomials enables them to be the stability polynomial choices~\cite{Koornwinder2013, RKL2, SKARAS_RKG2021109879}. By derivations in~\cite{RKL2, SKARAS_RKG2021109879}, the RKL1, RKL2, RKG1, and RKG2 methods, respectively, take the form of
\begin{align}
\label{eqn:RKL1}
u^{t+dt}=P_{s}\Big(I+\frac{2dt}{s^2+s}\mathbf{M}\Big)u^t;
\end{align}
\begin{align}
\label{eqn:RKL2}
u^{t+dt}=\Big[\Big(1-\frac{s^2+s-2}{2s(s+1)}\Big)I+\frac{s^2+s-2}{2s(s+1)}P_{s}\Big(I+\frac{4dt}{s^2+s-2}\mathbf{M}\Big)\Big]u^t;
\end{align}
\begin{align}
\label{eqn:RKG1}
u^{t+dt}=\frac{2}{(s+1)(s+2)}C^{(\lambda)}_{s}\Big(I+\frac{4dt}{s^2+3s}\mathbf{M}\Big)u^t;
\end{align}
\begin{align}
\label{eqn:RKG2}
u^{t+dt}=\Big[\Big(1-\frac{2(s-1)(s+4)}{3s(s+3)}\Big)I+\frac{4(s-1)(s+4)}{3s(s+1)(s+2)(s+3)}C^{(\lambda)}_{s}\Big(I+\frac{6dt}{(s+4)(s-1)}\mathbf{M}\Big)\Big]u^t,
\end{align} 
where $P_s(x)$ is the $s$-th degree Legendre polynomial and $C_s^{(\lambda)}(x)$ is the $s$-th degree Gegenbauer polynomial with parameter $\lambda$~\cite{Koornwinder2013}. In all of our experiments, $\lambda$ is chosen to be $\frac{3}{2}$ in order to be consistent with~\cite{SKARAS_RKG2021109879}.

To get a Runge-Kutta formulation of the RKL and RKL methods,~\cite{RKL2, SKARAS_RKG2021109879} used recurrence relationship of the Legendre and Gegenbauer polynomials. For example, the RKL2 method used the Legendre polynomial recurrence relationship
\begin{equation}
    sP_s(x)=(2s-1)xP_{s-1}(x)-(s-1)P_{s-2}(x).
\end{equation}
Along with the several initial known Legendre polynomials, after rearranging terms, one superstep of the $s$-stage RKL2 method can be written out as
\begin{align}
\label{eqn:RK-presentation}
\begin{split}
    &Y_0=u(t_0), \quad Y_1=Y_0+\Tilde{\mu}_1dt\mathbf{M}Y_0,\\
    &Y_j=\mu_jY_{j-1}+\nu_jY_{j-2}+(1-\mu_j-\nu_j)Y_0+\Tilde{\mu}_jdt\mathbf{M}Y_{j-1}+\Tilde{\gamma}_jdt\mathbf{M}Y_0, 2\le j\le s,\\
    &u(t_0+dt)=Y_s,
\end{split}
\end{align}
where the parameters are specified in~\cite{RKL2}.
Similar derivations work for the RKL1 method~\cite{RKL2} and the RKG methods~\cite{SKARAS_RKG2021109879} by using the inductive relation between the Legendre and Gegenbauer polynomials of different degrees.

Notice that although the formulas seem to be complicated, the method is entirely explicit and can be easily implemented with only three stages of the solution storage required. Moreover, the stability constraint is much more relaxed than \revision{a usual} explicit method: denote $dt_{expl}:=\frac{2}{\lambda_{\max}}$ where $\lambda_{\max}$ is the maximum eigenvalue of the operator $\mathbf{M}$; as shown in~\cite{RKL2}, the largest possible timestep is $dt_{\max}=dt_{expl}\frac{s^2+s-2}{4}$. The restriction is relaxed quadratically with respect to $s$, so that when we choose $s$ large enough, the restriction on $dt$ related to $dx$ is almost negligible. Moreover, it is simple to control the stage number $s$ in the implementation. These advantages lead to the adaptive use of $s$ in our complete algorithm.

\subsection{Numerical Discretization}
\label{subsec:numericalprocess}
Strang splitting~\cite{strangsplitting} along with the super-time-stepping method described above to solve the semilinear heat equation~\ref{eqn:semiheat}. \revision{One of the operators in the Strang splitting is chosen to be the discrete Laplacian with standard central difference approximation, while the other operator is chosen to be $u\mapsto u^p$}. The axisymmetric surface diffusion equation \ref{eqn:surfacediffusion} is solved by discretization of the right-hand side via the central difference approximation in all terms, and then use the method of lines to solve the PDE by the super-time-stepping method. In both cases, the adaptive mesh refinement strategy shown in Fig.~\ref{fig:meshrefinement} are used. More details are described in Appendix~\ref{app:amr} and~\ref{app:discretization}.


\section{Experiment results}
\label{sec:experiments}
We now present the computational results. We compare the accuracy and run time of our explicit super-time-stepping methods with the backward Euler implicit method. 
Various initial grid sizes are used to guarantee the robustness of the RKL and RKG methods. We present representative results here and more detailed experiment results can be found in the supplementary documents.

\subsection{Semilinear Heat Equation }
\label{subsec:semiheatexperiment}

We consider the semilinear heat equation~\ref{eqn:semiheat} for the cases $p=2$ and $p=3$. The spatial domain is $[-1, 1]$ with symmetry about the origin. In both cases, the Dirichlet boundary condition $u(\pm1, t)=0$ is used. The initial condition is $u(x, 0)=u_0(x)=\frac{10}{1-0.5\cos(\pi x)}-\frac{20}{3}$ for all cases. Periodic boundary condition or other non-negative even initial conditions can also work equally well for our numerical method, as long as the initial maximum value $\|u_0\|_{L^\infty}$ is not overly small or large. The discretization and the entire algorithm are discussed in detail in Section~\ref{subsec:numericalprocess}. The initial values of $dt$ and $dx$ satisfy $dt=\frac{1}{8}dx$, and we use initial $dx=\frac{1}{128}, \frac{1}{256}, \frac{1}{512}$ as test cases. The spatial grid sizes are chosen to be powers of 2, 
so that in their digital representation is exact. The initial stage number $s$ is 200, guaranteeing in all cases that $dt/dx^2$ are in the stability region of the super-time-stepping method.

\begin{figure}[ht!]
\vspace{-0.5em}
\centering
\begin{subfigure}{0.49\linewidth}
\captionsetup{justification=centering, margin=0.5cm}
\includegraphics[width=0.9\textwidth]{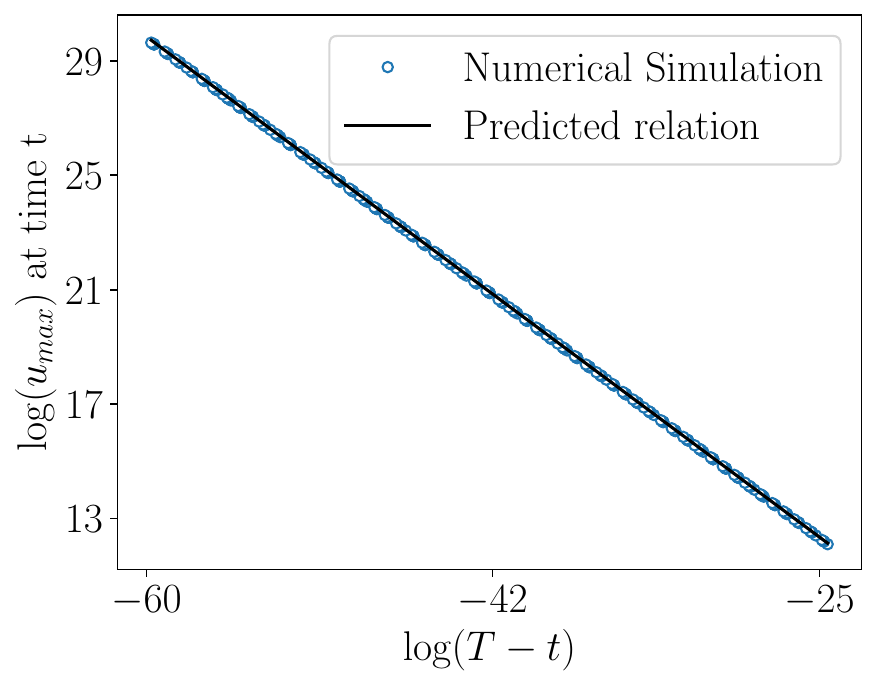}
\caption{RKL2 $\log$-$\log$ relationship between $T-t$ and maximum value}
\end{subfigure}
\begin{subfigure}{0.49\linewidth}
\captionsetup{justification=centering, margin=0.5cm}
\includegraphics[width=0.9\textwidth]{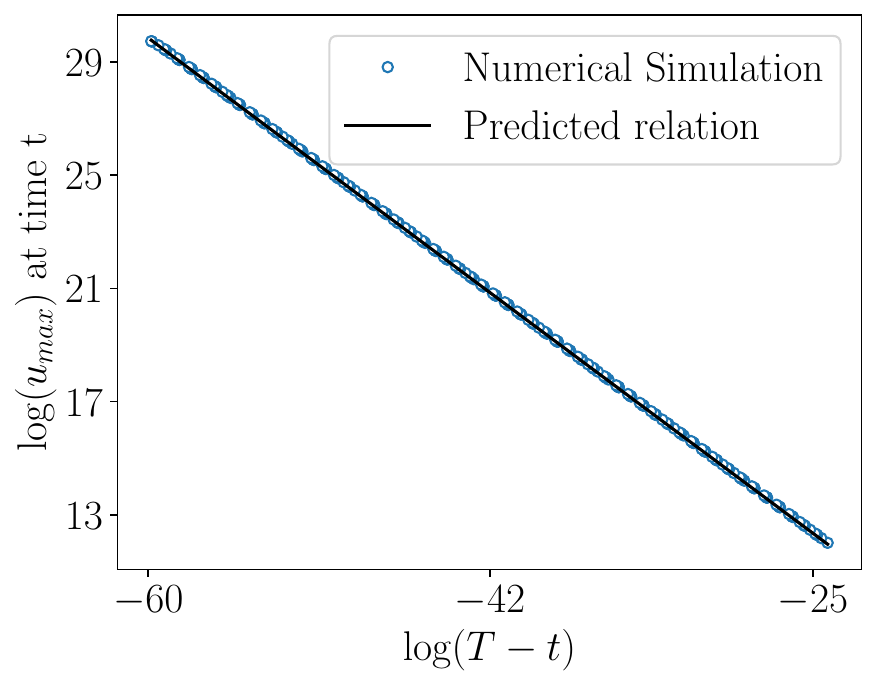}
\caption{RKG2 $\log$-$\log$ relationship between $T-t$ and maximum value}
\end{subfigure}
\vspace{-1em}
\caption{Log-log relationship between the $T-t$ and maximum value $\|u(\cdot, t)\|_{\infty}$, where $T$ is the theoretical blow-up time, with $p=3$. The calculated values fit perfectly with the predicted line, which has the expected slope $-\frac{1}{p-1}$.}
\label{fig:timevsheight}
\end{figure}

Fig.~\ref{fig:timevsheight} shows a $\log$-$\log$ relation between the time to blow-up and the maximum value; the predicted value comes from Equation~\ref{eqn:semiheat_max}. The plot is a straight line with slope $-\frac{1}{p-1}$ and horizontal intercept $-\frac{1}{p-1}\log(p-1)$. Since the exact blow-up time $T$ is unknown, we use the final time of our numerical simulation to approximate the theoretical blowing up time.


\begin{figure}[ht!]
\vspace{-0.5em}
\centering
\begin{subfigure}{0.49\linewidth}
\captionsetup{justification=centering, margin=0.15cm}
\includegraphics[width=0.9\textwidth]{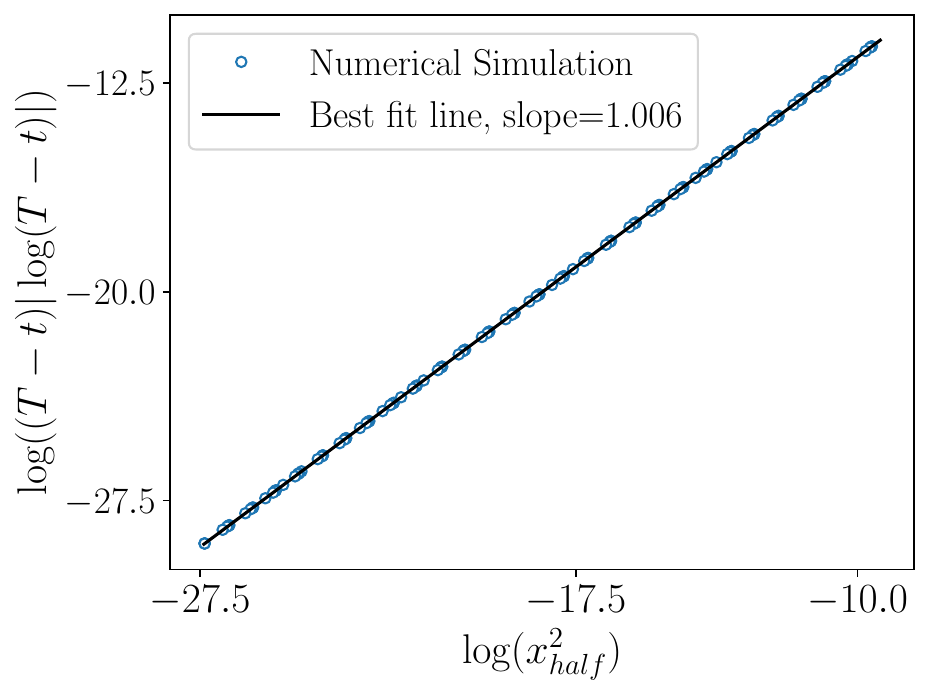}
\caption{RKL2 similarity variable scaling}
\end{subfigure}
\begin{subfigure}{0.49\linewidth}
\captionsetup{justification=centering, margin=0.15cm}
\includegraphics[width=0.9\textwidth]{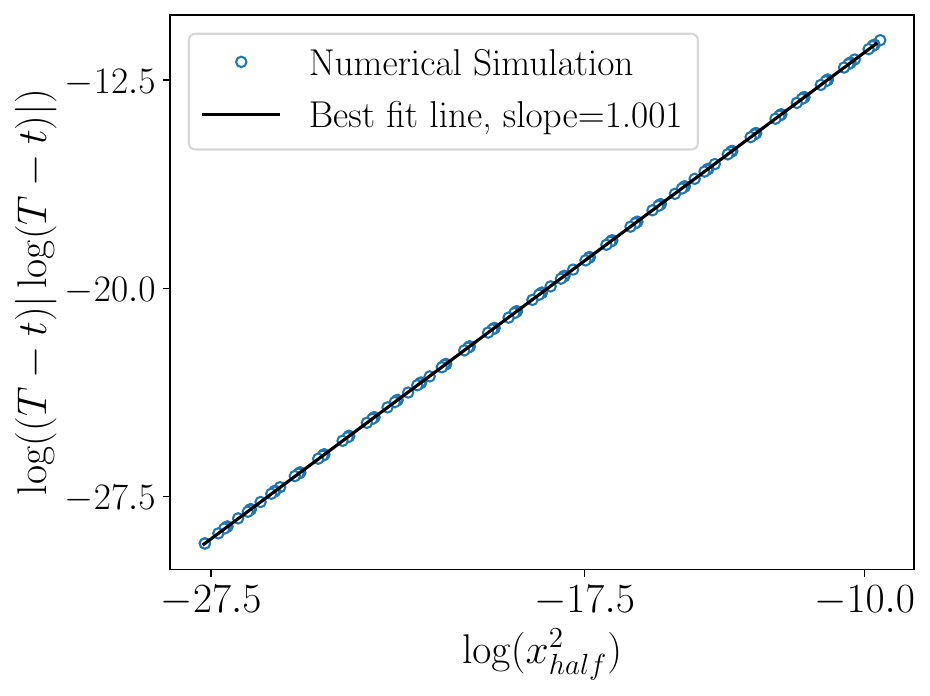}
\caption{RKG2 similarity variable scaling}
\end{subfigure}
\vspace{-1em}
\caption{Log-log relationship between the squared half-width and $(T-t)|\log(T-t)|$, with $p=3$. The value $x_{\text{half}}$ satisfies that $u(x_{\text{half}}, t)=\frac{1}{2}\|u(\cdot, t)\|_{\infty}$. The calculated values fit well with the predicted line, which has expected slope around $1$, indicating that the numerical solution accurately captures the scaling of the semilinear heat equation with the log correction.}
\label{fig:similarityscaling}
\end{figure}

Fig.~\ref{fig:similarityscaling} shows the $\log$-$\log$ relationship between the square of the numerically calculated half-width and the time variable quantity $(T-t)|\log(T-t)|$. The half-width provides a natural measurement of the spatial scaling of the blow-up. Our goal is to track the similarity variable $\frac{x^2}{(T-t)|\log(T-t)|}$. According the the similarity profile in Equation~\ref{eqn:semiheat_asymp}, we expect the plot to indicate a linear relationship with slope $1$. The results in Fig.~\ref{fig:similarityscaling} confirms our expectation, indicating that our numerical simulations accurately captures the similarity profile.

\begin{figure}[ht!]
\vspace{-0.5em}
\centering
\begin{subfigure}{0.49\linewidth}
\captionsetup{justification=centering, margin=0.15cm}
\includegraphics[width=0.9\textwidth]{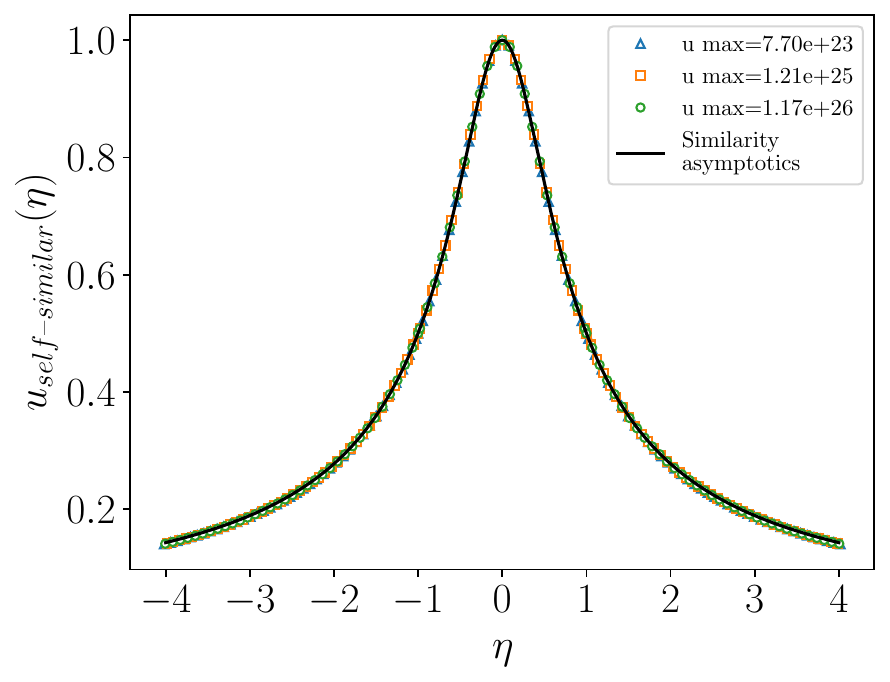}
\caption{RKL2 similarity solution}
\end{subfigure}
\begin{subfigure}{0.49\linewidth}
\captionsetup{justification=centering, margin=0.15cm}
\includegraphics[width=0.9\textwidth]{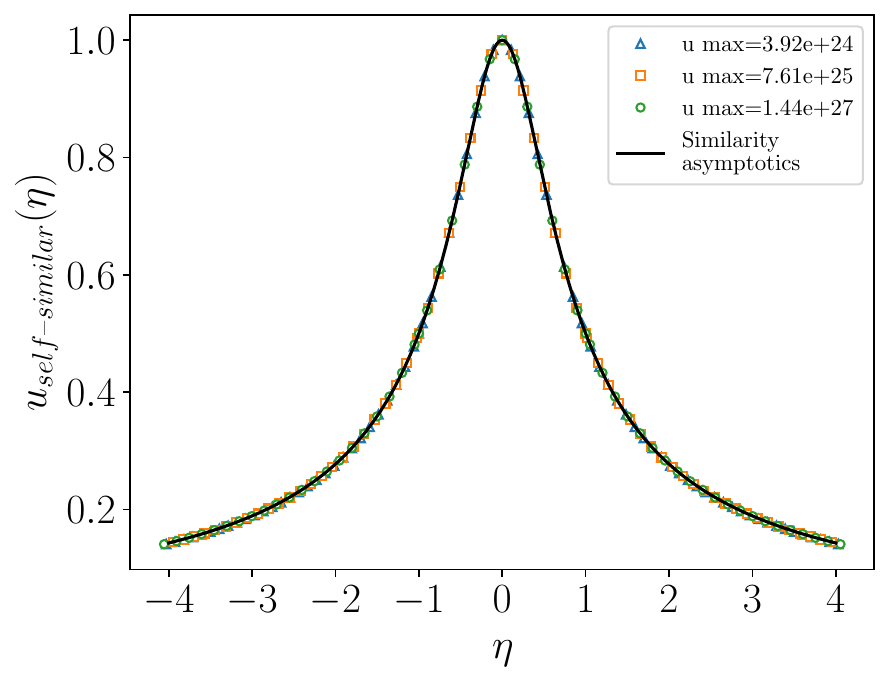}
\caption{RKG2 similarity solution}
\end{subfigure}
\vspace{-1em}
\caption{Rescaled RKL2 and RKG2 simulations for selected times comparing with the expected profile, with $p=3$. The x-axis is divided by $x_{\text{half}}$ where $u(x_{\text{half}}, t)=\frac{1}{2}\|u(\cdot, t)\|_{\infty}$, and the $u$ value is divided by $\|u(\cdot, t)\|_{\infty}$ at each selected time $t$. Similarity variable $\eta:=x/x_{\text{half}}$ is introduced. The solid line is the expected asymptotic profile, and the markers are the simulated values after the rescaling for selected times. The expected profile is given by Equation~\ref{eqn:similaritycollapse}.}
\label{fig:semilienar_rescaling}
\end{figure}

\begin{figure}[ht!]
\vspace{-0.5em}
\centering
\begin{subfigure}{0.49\linewidth}
\captionsetup{justification=centering, margin=0.5cm}
\includegraphics[width=0.9\textwidth]{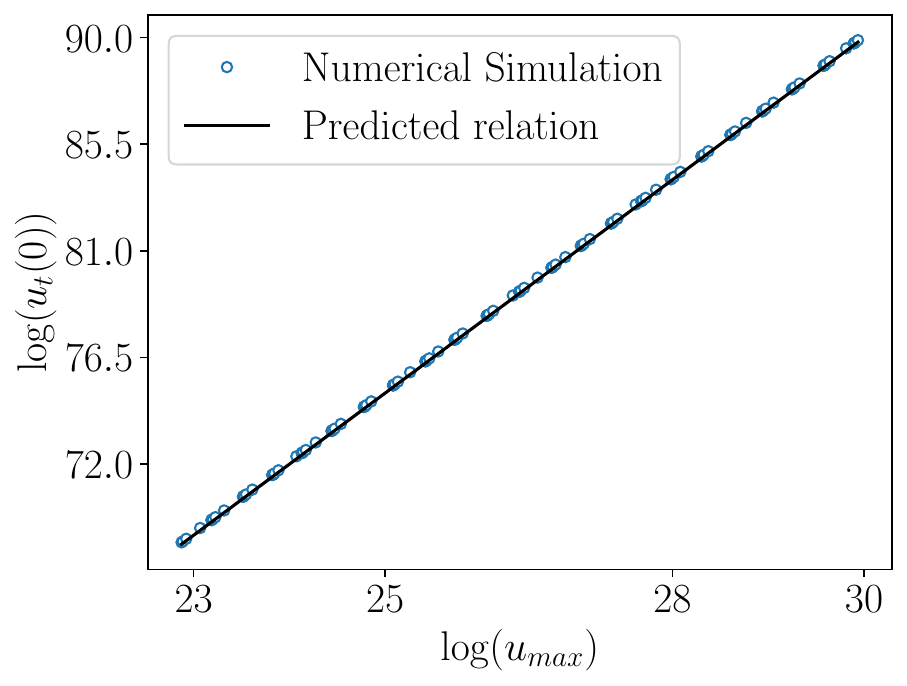}
\caption{RKL2 $\log$-$\log$ relationship between the maximum value and the time derivative at $x=0$}
\end{subfigure}
\begin{subfigure}{0.49\linewidth}
\captionsetup{justification=centering, margin=0.5cm}
\includegraphics[width=0.9\textwidth]{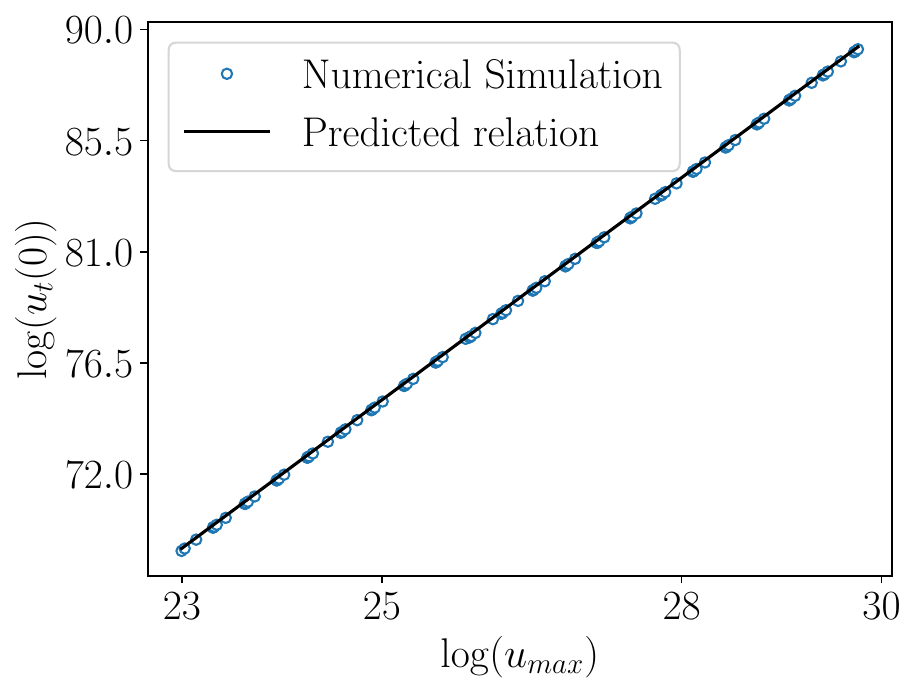}
\caption{RKG2 $\log$-$\log$ relationship between the maximum value and the time derivative at $x=0$}
\end{subfigure}
\vspace{-1em}
\caption{Log-log relationship between the maximum value $\|u(\cdot, t)\|_{\infty}$ and the time derivative of the solution at $x=0$, $p=3$. The time derivative is calculated numerically using forward difference approximation based on the numerical solution.}
\label{fig:timederivative}
\end{figure}

We investigate the spatial scaling in more detail. By the self-similarity of the asymptotic behavior, the shape of the solutions at different times should collapse onto a single curve up to a rescaling. We select several different times; at each time $t$, we divide $x$ by $x_{\text{half}}(t)$ and divide $u$ by $\|u(\cdot, t)\|_{L^\infty}$. According to previous calculations, the rescaled $u_{re}(x, t)$ takes the form of
\begin{equation}
\label{eqn:similaritycollapse}
    u_{re}(x, t)\sim(p-1)^{\frac{1}{p-1}}[p-1+(p-1)x^2(2^{p-1}-1)]^{-\frac{1}{p-1}}=\Big(\frac{1}{1+x^2(2^{p-1}-1)}\Big)^{\frac{1}{p-1}},
\end{equation}
which is time independent as expected. The rescaled solution is plotted in Fig.~\ref{fig:semilienar_rescaling}; in each subplot, three solutions at different times are plotted, each with maximum value above $10^{20}$. As can be seen in the plots, the markers all lie on the expected line provided by equation~\ref{eqn:similaritycollapse}.

Finally, we measure the accuracy for the first derivative of the solution. The predicted relationship in equation~\ref{eqn:semiheat_timederivative} is validated in Fig.~\ref{fig:timederivative}, by taking a $\log$-$\log$ plot of maximum value of the function. As shown in the plots, the data points perfectly fit the predicted relationship, with repeated piecewise patterns indicating the mesh refinement behavior. With all the plots above, we conclude that the super-time-stepping methods simulates well the semilinear heat equation over many orders of magnitude, showing the success of this fully explicit method applied to a stiff equation with a highly refined spatial grid.

Table~\ref{table:semiheattime}
shows the run time of our calculation compared to a semi-implicit time-stepping method, \revision{using the standard LU factorization linear system solver in Python}. Both $p=2$ and $p=3$ cases are tested, with three initial spatial grid sizes $dx=\frac{1}{128}, \frac{1}{256}, \frac{1}{512}$ and initial timestep $dt=\frac{1}{8}dx$ in each case. The same adaptive time-stepping and adaptive mesh refinement are used in all the experiments. We stop the simulation once the maximum value of the solutions reaches $10^{30}$.
The run time required for our explicit RKL and RKG methods are significantly less than the run time of the semi-implicit method.

\begin{table}[!t]
    \begin{center}
    \begin{tabular}{|c|c|c|c|c|}
    \hline
    \multicolumn{2}{|c|}{ }& \textit{dx=$\frac{1}{128}$} & \textit{dx=$\frac{1}{256}$} & \textit{dx=$\frac{1}{512}$}\\
    \hline
    \multirow{3}{*}{\textit{\textbf{p=2}}} 
    &\textit{Semi-implicit} & 1701.48s & 5471.50s & 21357.80s\\
    \cline{2-5}
    &\textit{RKL2} & 114.29s & 240.80s & 759.73s\\
    \cline{2-5}
    &\textit{RKG2} & \textbf{82.05}s & \textbf{177.41}s & \textbf{649.39}s\\
    \hline
    \multirow{3}{*}{\textit{\textbf{p=3}}} 
    &\textit{Semi-implicit} & 4482.14s & 19139.95s & 74925.77s\\
    \cline{2-5}
    &\textit{RKL2} & \textbf{202.21}s & \textbf{465.41}s & \textbf{1021.62}s\\
    \cline{2-5}
    &\textit{RKG2} & 299.55s & 754.21s & 1597.65s\\
    \hline
    \end{tabular}
    \end{center}
    \caption{Run time comparison between semi-implicit method and the stabilized Runge--Kutta methods for the semilinear heat equation. The fastest method in each case is bold. We stop the simulation when the solution blows up to $10^{30}$ both $p=2$ and $p=3$. All the stabilized Runge-Kutta simulations significantly outperform the classical implicit-explicit method, by at least more than 20 times.}
    
    \label{table:semiheattime}
\end{table}

\subsection{Surface Diffusion Equation}
Simulations of the surface diffusion equation~\ref{eqn:surfacediffusion} are conducted. First we verify the cone angle near the pinchoff. As mentioned in~\cite{bernoff1998axisymmetric}, the cone angle of the first analytic similarity solution is $46.0444^\circ$. \revision{This corresponds to the first derivative of the Eulerian variable solution near the pinchoff point and also to the far field derivative of the similarity solution, yielding a slope of approximately 1.037 in absolute value}. The spatial domain is chosen to be $[-2\pi, 2\pi]$ and periodic boundary conditions are used. The initial condition is  $u(x, 0)=-1.8(\frac{1}{4}\cos(\frac{x}{2})+\frac{1}{4})+1.2$, so that it is symmetric about the origin with moderately small minimum value. We test our numerical method with initial $dx=\frac{4\pi}{128}, \frac{4\pi}{256}, \frac{4\pi}{512}$ and initial $dt=10^{-5}$ for all the cases. Two and one initial mesh refinements are applied the first two cases respectively before the simulation for more accurate simulation near the singularity.

\begin{figure}[ht!]
\vspace{-0.5em}
\centering
\begin{subfigure}{0.49\linewidth}
\captionsetup{justification=centering, margin=1cm}
\includegraphics[width=0.9\textwidth]{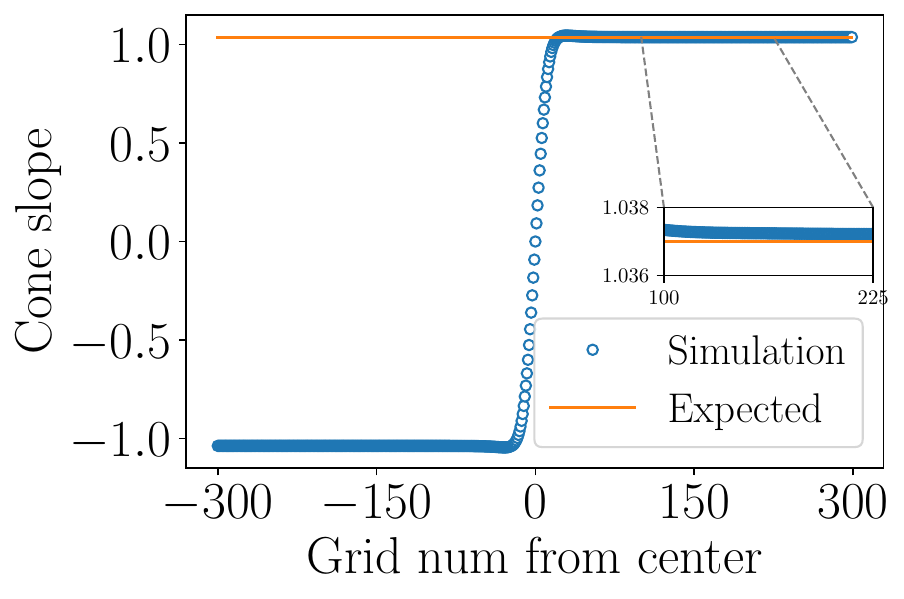}
\caption{RKL2 cone slope with minimum value $10^{-8}$}
\end{subfigure}
\begin{subfigure}{0.49\linewidth}
\captionsetup{justification=centering, margin=1cm}
\includegraphics[width=0.9\textwidth]{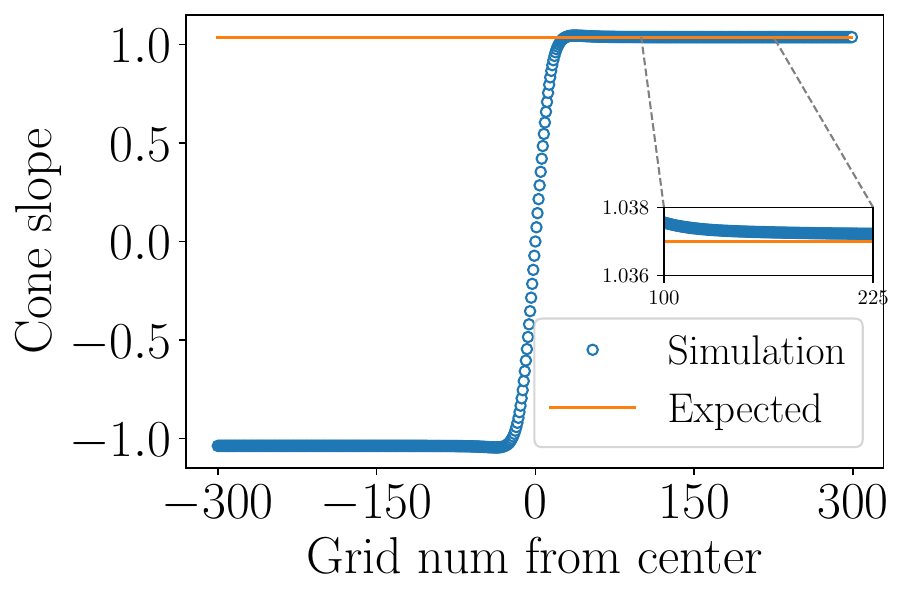}
\caption{RKL2 cone slope with minimum value $10^{-10}$}
\end{subfigure}
\vspace{-1em}
\caption{Surface diffusion equation cone angle test for the RKL2 simulations, with minimum value of $r(t)$ reaches to around $10^{-8}$, and $10^{-10}$. Initial $dx=\frac{4\pi}{512}$. The numerical first derivative is computed using standard forward difference; the horizontal orange line indicates the theoretically expected first derivative in absolute value.}
\label{fig:surfdiffangle}
\end{figure}

\begin{figure}[ht!]
\vspace{-0.5em}
\centering
\begin{subfigure}{0.49\linewidth}
\captionsetup{justification=centering, margin=1cm}
\includegraphics[width=0.9\textwidth]{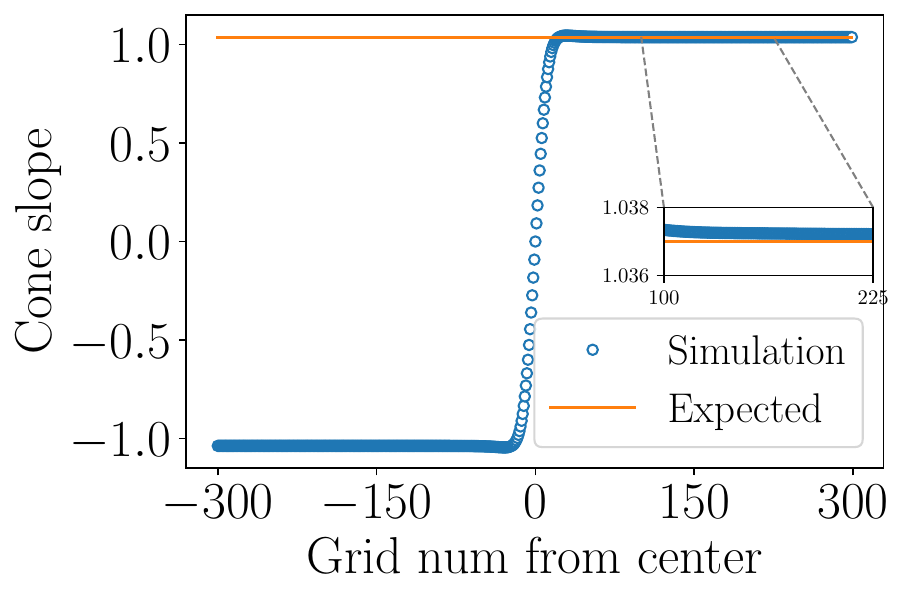}
\caption{RKG2 cone slope with minimum value $10^{-8}$}
\end{subfigure}
\begin{subfigure}{0.49\linewidth}
\captionsetup{justification=centering, margin=1cm}
\includegraphics[width=0.9\textwidth]{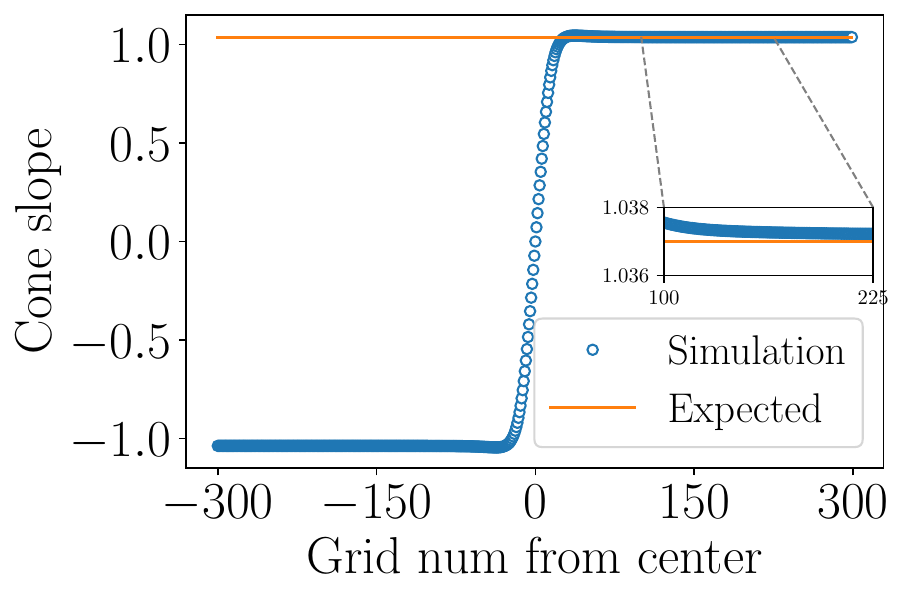}
\caption{RKG2 cone slope with minimum value $10^{-10}$}
\end{subfigure}
\vspace{-1em}
\caption{Surface diffusion equation cone angle test for the RKG2 simulations, under the same condition as Fig.~\ref{fig:surfdiffangle}. Initial $dx=\frac{4\pi}{512}$. Both RKL2 and RKG2 methods achieve similar results; see Fig.~\ref{fig:surfdiffangle} for more details.}
\label{fig:surfdiffangle_RKG}
\end{figure}

\begin{figure}[ht!]
\vspace{-0.5em}
\centering
\begin{subfigure}{0.49\linewidth}
\captionsetup{justification=centering, margin=1cm}
\includegraphics[width=0.9\textwidth]{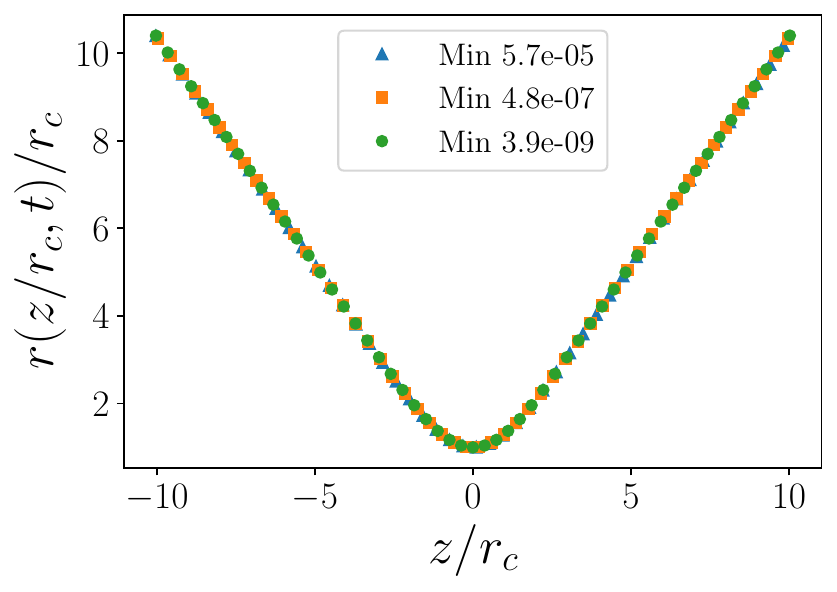}
\caption{RKL2 similarity solution}
\end{subfigure}
\begin{subfigure}{0.49\linewidth}
\captionsetup{justification=centering, margin=1cm}
\includegraphics[width=0.9\textwidth]{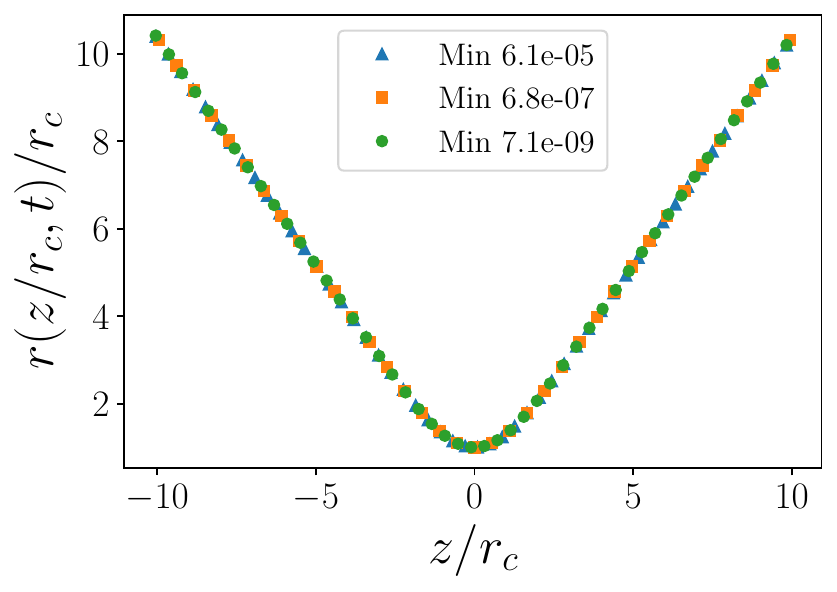}
\caption{RKG2 similarity solution}
\end{subfigure}
\vspace{-1em}
\caption{Rescaled solutions of the RKL2 and RKG2 surface diffusion equation simulations, with initial $dx=\frac{4\pi}{512}$. For each simulation, solutions with various minimum values ranging from $10^{-5}$ to $10^{-10}$ are selected; the horizontal axis is divided by the corresponding minimum function value, while the vertical axis is also rescaled by first using the modified horizontal axis and then divided by the corresponding minimum function value. As shown, the numerical solutions collapse into a single curve (not all grid points are shown in order to make all the time intervals visible).}
\label{fig:surfdiff_rescaling}
\end{figure}

To precisely measure the pinchoff behavior, we plot the cone slopes near the pinchoff center $z=0$ as shown in Fig.~\ref{fig:surfdiffangle}. The cone slopes are shown for the RKL2 simulations up to $10^{-8}$ and $10^{-10}$, respectively, with initial $dx=\frac{4\pi}{128}, \frac{4\pi}{256}, \frac{4\pi}{512}$. The orange line represents the expected cone slope in absolute value. As can be seen in the graphs, in all cases the simulated cone slopes match the expected pinchoff angle slope exactly as theoretically predicted, showing that the super-time-stepping methods are able to correctly and accurately capture the physical similarity solution as derived in~\cite{bernoff1998axisymmetric}. The rescaled solutions plotted in Fig.~\ref{fig:surfdiff_rescaling} further verify that the similarity solution is successfully simulated by the numerical solution; as shown in the figure, numerical simulations at various times with various minimum values are rescaled, and the rescaled solutions overlap with each other.

Another test is to verify that the speed of the numerical solution growth matches the predicted growth speed as indicated in Section~\ref{sec:pdeproperty}. Fig.~\ref{fig:surfdifftimederi} shows the $\log$-$\log$ relation between the absolute value of time derivative of the minimum value and the minimum value of $r$; the data points perfectly match the predicted relation, without any shifting in the $x$-interception or misalignment in the slope.

Table~\ref{table:surfdifftime} shows the run time of stabilized Runge--Kutta methods in comparison with the implicit backward Euler method. We use three initial spatial grid sizes $dx=\frac{4\pi}{128}, \frac{4\pi}{256}, \frac{4\pi}{512}$ and initial timestep $dt=10^{-5}$ in all cases. The same adaptive time-stepping and adaptive mesh refinement are used in all the experiments. We stop the simulation when the minimum value of the solutions reaches $10^{-10}$ for the RKL and RKG methods, and we terminate the simulation for the backward Euler when minimum reaches $2\times10^{-4}$ for time saving. The run time required for the explicit RKL and RKG methods are significantly less than the run time of the backward Euler method. 

\begin{figure}[ht!]
\vspace{-0.5em}
\centering
\begin{subfigure}{0.49\linewidth}
\captionsetup{justification=centering}
\includegraphics[width=0.9\textwidth]{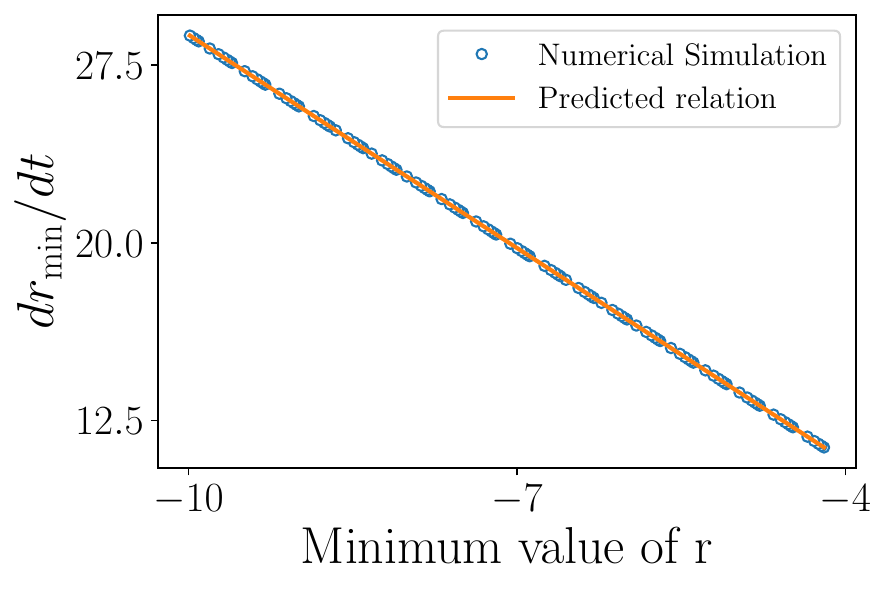}
\caption{RKL2 $\log$-$\log$ relationship between minimum value and time derivative}
\end{subfigure}
\begin{subfigure}{0.49\linewidth}
\captionsetup{justification=centering}
\includegraphics[width=0.9\textwidth]{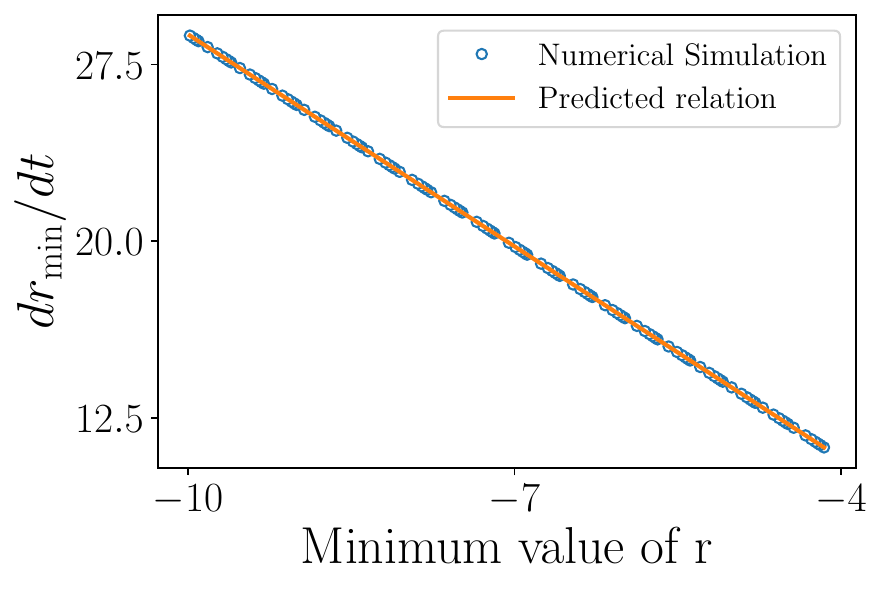}
\caption{RKG2 $\log$-$\log$ relationship between minimum value and time derivative}
\end{subfigure}
\vspace{-1em}
\caption{Log-log relationship between the absolute value of time derivative of the minimum value and the minimum value of $r$, surface diffusion equation. The time derivative is numerically calculated using the standard forward difference. According to the theoretical prediction, after taking the logarithm the absolute value of the time derivative has slope $-3$.}
\label{fig:surfdifftimederi}
\end{figure}

\begin{table}[!t]
    \begin{center}
    \begin{tabular}{|c|c|c|c|}
    \hline
    { }& \textit{dx=$\frac{4\pi}{128}$} & \textit{dx=$\frac{4\pi}{256}$} & \textit{dx=$\frac{4\pi}{512}$}\\
    \hline
    \textit{Implicit Euler} &  3608.68s & 8037.48s & 19576.61s\\
    \hline
    \textit{RKL2} & 47.28s & 63.51s & 127.88s\\
    \hline
    \textit{RKG2} & 83.69s & 96.79s & 154.20s\\
    \hline
    \textit{RKL1} & \textbf{24.41}s & \textbf{38.38}s & \textbf{45.35}s\\
    \hline
    \textit{RKG1} & 34.13s & 50.39s & 99.21s\\
    \hline
    \end{tabular}
    \end{center}
    
    \caption{Run time comparison between implicit Euler's method and the super-time-stepping explicit method for the surface diffusion equation. For each initial $dx$, the fastest method is highlighted. We stop the simulation when the solution blows up to $2\times10^{-4}$. All the super-time-stepping method simulations require significantly less computational time than the classical implicit Euler's method.}
    \label{table:surfdifftime}
\end{table}

\section{Proof of Monotone Stability}
\label{sec:monotoneproof}
Despite the success of super-time-stepping methods, 
few theoretical guarantees have been proved. In~\cite{RKL2} and~\cite{OSULLIVAN_RKG2019209}, the consistency and stability of both the RKL methods and the RKG methods have been established for an arbitrary number of stages $s$. 
However, until now rigorous monotone stability was still an open question. We prove this result for the (semi-)linear heat equation with boundaryless/periodic boundary conditions. 

A numerical method is called ``monotonically stable" if the for any two sets of initial data $u_0$ and $v_0$ with 
$$v_0(x)\ge u_0(x)\quad\forall x,$$
then the corresponding solutions satisfy
$$v(x, t)\ge u(x, t)\quad\forall x, t.$$ 
Monotone stability is a discrete analogy to the comparison principle of second-order parabolic equations, which may not satisfied by the fourth-order parabolic equation~\cite{bertozzi1995lubrication, Bertozzi1998THEMO}. Numerous physical systems and PDEs have the monotone property, for example the semilinear heat equation~\ref{eqn:semiheat}, and it is essential for numerical methods to capture such a property in order to simulate the corresponding systems realistically. 

The RKL methods (\cite{RKL2}) were primarily developed to address non-monotonicity issues that arise in the Runge-Kutta-Chebyshev super-time-stepping method. 
However, to date, people have only numerically verified the RKL monotone stability up to a specific number of stages $s$ ($s\le64$). Other investigations~\cite{DAWES2021104762, SKARAS_RKG2021109879} have numerically explored the monotone stability under different types of boundary conditions, but still without rigorous proof. Later, the RKG method was proposed in~\cite{SKARAS_RKG2021109879} to ensure monotone stability under the Dirichlet boundary condition, but still only numerical verifications were performed.

In this section, we present a proof of monotone stability for both the RKL and the RKG methods applied to the linear heat equation and semilinear heat equation, under no boundary conditions or equivalently periodic condition, within their numerical stability constraint.

\subsection{Preliminary Definitions and Setup}
For a semi-discretized PDE in the form of 
$$\frac{du}{dt}=Mu$$
where $M$ contains the spatial discretization, we write the explicit one-step numerical method in the form of 
$$u_i^{t+dt}=\mathcal{H}(u^t;j).$$
In RKL and RKG methods, $\mathcal{H}(u^t; j)$ involves multiple grid points around the $i$th spatial grid, and the number of points becomes larger as $s$ increases. Using this notation, monotone stability can be rigorously formulated below:
\begin{definition}
\label{def:monotonicity}
A numerical method is called ``monotonically stable", or simply ``monotone", if and only if 
\begin{equation}
\label{eqn:monotonicity}
\frac{\partial\mathcal{H}}{\partial u_j^t}(u^t;j)\ge0
\end{equation}    
for all $j$ in the grid points at timestep $t$. 
\end{definition}
More details and concrete examples can be found in~\cite{leveque, RKL2} and Appendix~\ref{app:monotone_example}. 

In the rest of the proof in this section, we exclusively focus on the boundaryless one-dimensional heat equation $u_t=\alpha u_{xx}$, numerically solved by RKL and RKG methods with the method of lines using the standard spatial central difference. More precisely, we use the RKL and RKG methods to solve for
\begin{equation}
\label{eqn:theoryheat}
    \frac{du}{dt}=\alpha Mu,\quad t\in(0, T]\quad \alpha\in\R;\quad u(0)=u_0,
\end{equation}
where $\alpha\in\R_+$, $u\in\R^\Z$, $M:\R^\Z\to\R^\Z$ is defined by $[Mu]_i:=\frac{[u]_{i+1}-2[u]_{i}+[u]_{i-1}}{dx^2}$, $[u]_i$ is the $i$-th entry of $u\in\R^\Z$, $i\in\Z$. We denote the CFL coefficient by $c:=\alpha\frac{dt}{dx^2}$, where $dt$ is the super timestep size. Notice that all the results below also work for the periodic \revision{and homogeneous Neumann boundary condition cases}, since the periodic boundary condition problem can be equivalently seen as a boundaryless problem with repeated values, \revision{and the homogeneous Neumann boundary condition problem can be seen as a periodic boundary condition problem with each period to be the original domain reflected and doubled}.

\subsection{Runge--Kutta--Legendre Methods Monotonicity}


The difficulty of the proof lies in controlling the Legendre polynomial and in the repetitive application of the central difference operator $M$ on the grid values. It is natural to start with an induction attempt given the recurrence relation of the Legendre polynomial as in the derivation of the RKL methods. However, for the RKL methods, there is barely any relationship between the $s$-stage method and the $(s+1)$-stage method. For example, for the RKL1 method, $P_{s+1}\Big(1+\frac{2}{(s+1)^2+(s+1)}x\Big)$ and $P_{s}\Big(1+\frac{2}{s^2+s}x\Big)$ cannot be directly linked by the Legendre polynomial recurrence relation given the different arguments in the Legendre polynomial. Another approach might be to write an RKL method as the composition of Euler's method with different timestep lengths~\cite{RKL2}, and try to prove the monotonicity of each Euler's method. Yet each of the internal Euler's step might not even be stable, thus leading to no hope of its monotone stability, which in general is a stronger property than the numerical stability. We also note that Shu-Osher~\cite{Shu-Osher} represented the classical RK methods in a similar formulation to the RKL and RKG method. They proposed a simple criterion for the RK method to be total-variation diminishing (TVD)~\cite{leveque}. However, with more careful observation, it can be seen that the Shu-Osher representation cannot be directly applied to the RKL or the RKG methods, since it requires the coefficients before the intermediate steps $Y_j$ in~\ref{eqn:RK-presentation} to be all non-negative, a property that is neither satisfied by the RKL nor the RKG method; moreover, the TVD property is a weaker condition than our desired monotone stability~\cite{leveque}.

Our proof utilizes a more direct approach, based on an alternating summation representation of the Legendre polynomial and several properties of the generalized hypergeometric functions. This approach directly leads to the result that the monotonicity conditions for both RKL1 and RKL2 methods are exactly the same as their stability constraints, respectively, for the constant coefficient heat equation and semilinear heat equation without boundary condition. 

\subsubsection{Combinatorics Formulation}

In this section, we formulate the crucial combinatorial statement, whose validity is equivalent to both the monotonicity of the RKL1 and the RKL2 methods.

\begin{lemma}
\label{lemma:rklcondition}
The \revision{$s$-stage} RKL1 method is monotone under its stability constraint
$$c\le\frac{s^2+s}{4}$$
when applied to problem~\ref{eqn:theoryheat} if and only if 
\begin{equation}
\label{eqn:combformula_rkl1}
\sum_{k=j}^s\binom{s}{k}\binom{s+k}{k}x(s)^k\binom{2k}{j+k}(-1)^{j+k}\ge0    
\end{equation}
holds for any $s\in\N$ and $0\le j\le s$, and $0\le x(s)\le\frac{1}{4}$.
\end{lemma}
\begin{proof}
The Legendre polynomial of $s$-th order $P_s(x)$ can be written as~\cite{Koornwinder2013}
$$P_s(x)=\sum\limits_{k=0}^s\binom{s}{k}\binom{s+k}{k}\Big(\frac{x-1}{2}\Big)^k.$$
Plugging into the RKL1 method~\ref{eqn:RKL1}, we have that
$$u_i^{t+dt}=\sum_{k=0}^s\binom{s}{k}\binom{s+k}{k}\Big(\frac{1}{s^2+s}\Big)^kdt^k[M^ku^t]_i$$
where $[M^ku^t]_i$ denotes the $i$-th element of the vector $M^ku^t$.
Using induction, it can be shown that
$$[M^ku^t]_i=\frac{1}{dx^{2k}}\sum_{j=0}^{2k}\binom{2k}{j}(-1)^ju_{i+j-k}^t=\frac{1}{dx^{2k}}\sum_{j=-k}^{k}\binom{2k}{j+k}(-1)^{j+k}u_{i+j}^t.$$
so that
$$u_i^{t+dt}=\sum_{k=0}^n\binom{s}{k}\binom{s+k}{k}\Big(\frac{1}{s^2+s}\Big)^kc^k\sum_{j=-k}^{k}\binom{2k}{j+k}(-1)^{j+k}u_{i+j}^t.$$
By definition~\ref{eqn:monotonicity}, the monotonicity holds if and only if $\frac{\partial u_i^{t+dt}}{\partial u_j^{t}}\ge0$ for all grid indices $i, j$; a change of index enables us to focus on each grid index separately:
$$u_i^{t+dt}=\sum_{j=-s}^s\sum_{k=|j|}^s\binom{s}{k}\binom{s+k}{k}\Big(\frac{1}{s^2+s}\Big)^kc^k\binom{2k}{j+k}(-1)^{j+k}u_{i+j}^t.$$
By symmetry, only the cases where $j\ge0$ needs to be considered. More precisely, the $s$-stage RKL1 monotonicity is satisfied under its stability constraint if and only if
$$\sum_{k=j}^s\binom{s}{k}\binom{s+k}{k}\Big(\frac{1}{s^2+s}\Big)^kc^k\binom{2k}{j+k}(-1)^{j+k}\ge0$$
for any $s\in\N$, $0\le j\le s$ and $0\le c\le\frac{s^2+s}{4}$. This proves the lemma.
\end{proof}

\begin{lemma}
The \revision{$s$-stage} RKL2 method is monotone under its stability constraint
$$c\le\frac{s^2+s-2}{8}.$$
if and only if the \revision{$s$-stage} RKL1 method is monotone under its stability constraint 
$$c\le\frac{s^2+s}{4}.$$
\end{lemma}
\begin{proof}
Using the same derivation as in Lemma~\ref{lemma:rklcondition}, the RKL2 method is monotone under its stability constraint if and only if 
\begin{equation}
\label{eqn:combformula_rkl2}
\begin{cases}
1-\frac{s^2+s-2}{2s(s+1)}+\frac{s^2+s-2}{2s(s+1)}\sum_{k=j}^s\binom{s}{k}\binom{s+k}{k}x(s)^k\binom{2k}{j+k}(-1)^{j+k}\ge0 & j=0\\
\frac{s^2+s-2}{2s(s+1)}\sum_{k=j}^n\binom{s}{k}\binom{s+k}{k}x(s)^k\binom{2k}{j+k}(-1)^{j+k}\ge0 & j\neq0
\end{cases}
\end{equation}
holds for any $s\in\N$ and $0\le j\le s$, and $0\le x(s)\le\frac{1}{4}$. Clearly $f(s):=\frac{s^2+s-2}{2s(s+1)}$ is monotonically increasing and positive for $n\ge1$ by direct calculation; since $\underset{s\to\infty}{\lim}f(s)=\frac{1}{2}$, we can conclude that $0\le f(s)\le\frac{1}{2}$, so that equation~\ref{eqn:combformula_rkl2} holds given that equation~\ref{eqn:combformula_rkl1} holds. This concludes our proof.
\end{proof}

\subsubsection{Estimation of the Alternating Sum and Conclusion}

\begin{definition}[Generalized hypergeometric function~\cite{Duverney2024}]
A generalized hypergeometric function $\pFq{p}{q}{a_1,...,a_p}{b_1,...,b_q}{x}:\R^p\times\R^q\times\C\to\R$ is defined as
\begin{equation}
\label{eqn:hypergeo_def}
\pFq{p}{q}{a_1,...,a_p}{b_1,...,b_q}{x}:=\sum_{k=0}^\infty\frac{(a_1)_k...(a_p)_k}{(b_1)_k...(b_q)_k}\frac{x^k}{k!}
\end{equation}
where $\Gamma(x)$ is the gamma function, $(a)_k:=\frac{\Gamma(a+k)}{\Gamma(a)}$ is the Pochhammer symbol. 
\end{definition}

A straightforward observation is that we can simplify the Pochhammer symbol as $(a)_k:=a(a+1)\cdots(a+k-1)$ when $a\in\N$. The gamma function $\Gamma(x)$ used in the definition above is assumed with analytical continuation to the entire complex plane using the rule $\Gamma(x)=\frac{\Gamma(x+1)}{x}$, so that by classical complex analysis arguments it is a meromorphic function with isolated poles at the negative integers. Specifically, $\Gamma(x)$ is also well-defined for negative non-integer values.

The convergence of a hypergeometric function in special cases is important for the discussions below. Notice that by direct computation, when any of $a_1,...,a_p$ is a non-positive integer, the series degenerates into a finite sum, so that the hypergeometric function is always well-defined. Another observation is that when $p=q+1$, using the ratio test, the series is convergent for $|z|<1$ and divergent for $|z|>1$, so that we would only be able to guarantee the well-definedness for $|z|<1$.

Numerous transformations related to generalized hypergeometric functions have been explored intensively, including conditions that determine the non-negativity of them. One of the famous identities that is useful for our proof is as follows

\begin{theorem}[Clausen’s formula~\cite{Clausen}]
\label{thm:clausen}
For any $a, b\in\R$,
$$\pFq{3}{2}{2a, 2b, a+b}{a+b+1/2, 2a+2b}{x}=\Big(\pFq{2}{1}{a, b}{a+b+1/2}{x}\Big)^2.$$
holds for any $x\in\R$ when the generalized hypergeometric functions on both sides are convergent and well-defined.
\end{theorem}

In particular, it gives us a form of generalized hypergeometric functions that is always non-negative. Our goal is to write the alternating sum~\ref{eqn:combformula_rkl1} as a generalized hypergeometric function, with the hope that the coefficients are compatible with Clausen's formula and directly apply the theorem to conclude. Fortunately, this process indeed gives us the desired result using the calculations in the following lemma.


\begin{lemma}[Alternating sum estimation lemma]
\label{lemma:alterbinom}
For any $s\in\N$, $0\le j\le s$, and $x\equiv x(s)\in\R$ such that $0\le x\le\frac{1}{4}$,
$$\sum_{k=j}^s\binom{s}{k}\binom{s+k}{k}x^k\binom{2k}{j+k}(-1)^{j+k}\ge0.$$
\end{lemma}

A detailed proof is in Appendix~\ref{app:sum_lemma}. 

With the above lemma, our main theorem immediately follows using the condition in the Lemma~\ref{lemma:rklcondition}.

\begin{theorem}[Monotonicity of RKL1 and RKL2, heat equation]
\\
(a) The \revision{$s$-stage} RKL1 method~\ref{eqn:RKL1} applied to constant coefficient heat equation under periodic \revision{or homogeneous Neumann boundary condition} is monotone if and only if
$$\alpha\frac{dt}{dx^2}\le\frac{s^2+s}{4}.$$
(b) The \revision{$s$-stage} RKL2 method~\ref{eqn:RKL2} applied to constant coefficient heat equation under periodic \revision{or homogeneous Neumann boundary condition} is monotone if and only if
$$\alpha\frac{dt}{dx^2}\le\frac{s^2+s-2}{8}.$$
Note that both constraints above are the same as the linear stability constraints.
\end{theorem}

A direct generalization of the above result leads to the monotonicity of RKL methods when applied to the semilinear heat equation $u_t=\alpha u_{xx}+u^p$, under the boundaryless/periodic boundary condition case, solved using the Strang splitting.

\begin{corollary}[Monotonicity of RKL1 and RKL2, semilinear heat equation]
(a) The \revision{$s$-stage} RKL1 method~\ref{eqn:RKL1} applied to constant coefficient semilinear heat equation with Strang splitting is monotone if and only if it satisfies the stability constraint
$$\alpha\frac{dt}{dx^2}\le\frac{s^2+s}{4},$$
under periodic \revision{or homogeneous Neumann boundary condition}.
\\
(b) The \revision{$s$-stage} RKL2 method~\ref{eqn:RKL2} applied to constant coefficient semilinear heat equation with Strang splitting is monotone if and only if it satisfies the stability constraint
$$\alpha\frac{dt}{dx^2}\le\frac{s^2+s-2}{8},$$
under periodic \revision{or homogeneous Neumann boundary condition}.
\end{corollary}
\begin{proof}
Notice that the analytical solution of
$$u'(t)=u^p, u(t_0)=u_0$$
is $u(t)=\Big[t(1-p)+u_0^{1-p}-t_0(1-p)\Big]^{\frac{1}{1-p}}$, and for any fixed $t, p, t_0$ is clearly monotone in $u_0$ when $u_0$ is large. Along with the theorem above, the desired result follows.
\end{proof}

\subsection{Runge--Kutta--Gegenbauer Methods Monotonicity}
With the above RKL monotonicity results, it is natural to extend the proof procedure to the RKG methods, with the hope that some existing hypergeometric series transformations and non-negativity theorems may lead to the desired result. However, the same direct computation by replacing $P_s(z)$ with $C_s^{(3/2)}(z)$ in the above leads to the hypergeometric series $C\pFq{3}{2}{s+j+3, j+\frac{1}{2}, j-s}{2j+1, j+2}{4x}$ with some constant $C\equiv C(s, j)$, which to our best knowledge does not align with the known hypergeometric series non-negativity identities such as the Clausen's identity in Theorem~\ref{thm:clausen} or any of the Askey-Gasper type inequalities~\cite{AskeyGasper, Gasper1977, Askey1974, 10.2748/tmj/1178241523}. Fortunately, a recurrence relationship between Gegenbauer polynomials $C_s^{(\lambda)}$ with different $\lambda$ can directly lead to the monotonicity for RKG1, with the fact that $C_s^{(1/2)}(z)=P_s(z)$ and that RKL methods are monotone. A similar argument as in the RKL methods case can directly let us conclude the RKG2 monotonicity from the RKG1 monotonicity.

The proof of the RKG1 monotonicity can be summarized as the following theorem:
\begin{theorem}
The \revision{$s$-stage} RKG1 method~\ref{eqn:RKG1} applied to constant coefficient heat equation under periodic \revision{or homogeneous Neumann boundary condition} is monotone if it satisfies its stability constraint
$$\alpha\frac{dt}{dx^2}\le\frac{s^2+3s}{8}.$$
\end{theorem}
\begin{proof}
By the recurrence relationship of the Gegenbauer polynomials (which can be verified using direct computations),
\begin{equation}
\label{eqn:gegenrelation}
C_s^{(\lambda)}(z)=zC_{s-1}^{(\lambda)}(z)+\frac{2\lambda+s-2}{2\lambda-2}C_s^{(\lambda-1)}(z).
\end{equation}
Also notice that $C_s^{(1/2)}(z)=P_s(z)$ is exactly the $s$-th degree Legendre polynomial. The $s$-stage RKL1 monotonicity is equivalent to the statement that $\Big[P_s\Big(I+x\Tilde{M}\Big)u(t)\Big]_i$ is a non-negative finite linear combination of $[u]_j$'s, where $[\Tilde{M}u]_i:=[u]_{i+1}-2[u]_{i}+[u]_{i-1}$, under the condition $0\le x\le\frac{1}{2}$.

Now we will show that the RKG1 method, with $\lambda=\frac{3}{2}$ as in the original RKG paper~\cite{SKARAS_RKG2021109879}, is monotonically stable using induction. $n=0$ and $n=1$ cases are trivial via direct computation. Now suppose the result is true for some $n-1\in\N$; then for the $n$ case,
$$C_s^{(\lambda)}(I+x\Tilde{M})=(I+x\Tilde{M})C_{s-1}^{(\lambda)}(I+x\Tilde{M})+(s+1)C_s^{(\lambda-1)}(I+x\Tilde{M}).$$
Each entry of $(s+1)C_s^{(\lambda-1)}(I+x\Tilde{M})u(t)$ is a non-negative finite linear combination of $[u]_j$'s by the RKL1 monotonicity; by the inductive hypothesis, each entry of $C_{s-1}^{(\lambda)}(I+x\Tilde{M})u(t)$ is also a non-negative finite linear combination of $[u]_j$'s. Finally, for $0\le x\le\frac{1}{2}$, each entry of $(I+x\Tilde{M})v$ is again a non-negative finite linear combination of $[v]_j$'s for any $v\in\R^\Z$; therefore, each entry of $(I+x\Tilde{M})C_{s-1}^{(\lambda)}(I+x\Tilde{M})u(t)$ is a non-negative finite linear combination of $[u]_j$'s. Taking the summation of two terms, each entry of $C_s^{(\lambda)}(I+x\Tilde{M})u$ is a non-negative finite linear combination of $[u]_j$'s; hence RKG1 is monotone for the desired CFL condition.
\end{proof}
\begin{remark}The above proof can be very easily generalized to show that RKG1 methods are monotone for $\lambda=\frac{k}{2}$ cases where $k\in\N$, using exactly the same argument as above.
\end{remark}

\begin{corollary}
The \revision{$s$-stage} RKG2 method~\ref{eqn:RKG2} applied to constant coefficient heat equation under periodic \revision{or homogeneous Neumann boundary condition} is monotone if it satisfies its stability constraint
$$\alpha\frac{dt}{dx^2}\le\frac{(s+4)(s-1)}{12}.$$
\end{corollary}
\begin{proof}
Notice that $\frac{2(s-1)(s+4)}{3s(s+3)}$ is monotonically increasing and $\underset{s\to\infty}{\lim}\frac{2(s-1)(s+4)}{3s(s+3)}=\frac{2}{3}<1$. Thus, the same argument of extending RKL1 monotonicity to RKL2 monotonicity can be applied here.
\end{proof}

Finally, we can utilize the argument as in the RKL case to conclude that same for the semilinear heat equations as a corollary:
\begin{corollary}[Monotonicity of RKG1 and RKG2, semilinear heat equation]
(a) The \revision{$s$-stage} RKG1 method~\ref{eqn:RKG1} applied to constant coefficient semilinear heat equation with Strang splitting is monotone if it satisfies its stability constraint
$$\alpha\frac{dt}{dx^2}\le\frac{s^2+3s}{8},$$
under periodic \revision{or Neumann boundary condition}.
\\
(b) The \revision{$s$-stage} RKG2 method~\ref{eqn:RKG2} applied to constant coefficient semilinear heat equation with Strang splitting is monotone if it satisfies its stability constraint
$$\alpha\frac{dt}{dx^2}\le\frac{(s+4)(s-1)}{12},$$
under periodic \revision{or Neumann boundary condition}.
\end{corollary}

\section{Conclusions}
\label{sec:conclusions}

In this paper, the RKL and RKG super-time-stepping explicit methods are used to simulate both second-order and fourth-order spatial one-dimensional PDEs with finite-time singularities. Both methods excellently capture the asymptotic behavior and the scalings of the PDEs of many orders of magnitude, with greater computational efficiency than those of the widely used implicit methods. The numerical monotonicity of both methods without boundary condition is theoretically proved for the linear and semilinear heat equation.

One natural question is whether such methods can be used for all finite-time singularities, or are there restrictions.  
We conjecture that singularities with space-time scaling that does not correspond to the stability scaling of these schemes would require a different method.  One such example are the second-kind similarity solutions identified in fourth-order lubrication equations of the form $u_t + (u^nu_{xxx})_x=0$
for which the singularity has second-order parabolic scaling for the structure of the singularity when $n$ is small \cite{Symmetric}, however the fourth-order derivative requires a smaller timestep constraint for numerical stability using any forward time-stepping method.

For future work, the super-time-stepping explicit methods can be applied to higher dimensional stiff equations, where solving large systems of equations using implicit methods is computationally expensive; both time and memory costs should be largely reduced. Theoretically, it is still unknown how to prove that the RKG methods are monotone under the Dirichlet boundary condition, while the RKL methods are not. Moreover, monotonicity for variable-coefficient parabolic equations is another open question that needs a proof for theoretical guarantee. Such proofs might lead us to a better understanding of hypergeometric functions-based stabilized Runge-Kutta methods and motivate the design of new methods belonging to this class.

\appendix

\section{Adaptive Mesh Refinement}
\label{app:amr}
We use adaptive time-stepping and mesh refinement in our numerical simulations, for both the implicit and the explicit methods. We choose initial conditions that are symmetric and lead to a singularity at $x=0$; our numerical method can be easily modified and used in the same manner if the singularity is shifted or there are multiple singularity points (see e.g. \cite{BBDK}).  The initial mesh is a uniform grid. The mesh refinement criterion is determined by a second parameter. For the semilienar heat equation, we choose it to be the half-width of the numerical solution $x_{\text{half}}$, as defined in Section~\ref{semiheat_property}, which satisfies $u(x_{\text{half}}, t)=\frac{1}{2}\|u(\cdot, t)\|_{L^\infty}$; for the surface diffusion equation, the criterion is chosen to be the minimum value of the numerical simulation. 
\revision{In both cases the criteria is chosen to match the known scaling of the singularity. The half-width criteria is universal in that it always guarantees a certain number of grid points to resolve the inner spatial structure of the solution}. At each timestep, the half-width is calculated exactly using the existing grid point values with cubic spline interpolation. We first calculate $x_{\text{half}}$ for the initial condition and denote it by $x_{\text{half}}(t_0)$ by $t_0=0$. When time reaches $t_1>t_0$ such that $x_{\text{half}}(t_1)\le\frac{1}{2}x_{\text{half}}(t_0)$, we divide the middle half of the initial grid into another half. Then when time reaches $t_2>t_1$ such that $x_{\text{half}}(t_2)\le\frac{1}{2}x_{\text{half}}(t_1)$, we further divide the middle half of the first time refined meshes into half. Repeating this process, we obtain finer and finer grids near and symmetric about the origin. The refinement is done dynamically as the singularity evolvs. A similar procedure is done for the surface diffusion equation. Fig.~\ref{fig:meshrefinement} shows the mesh refinement for two times. Such an adaptive mesh refinement strategy is easy to implement; it also preserves the property that the distance between each two consecutive grids is in the form of $2^{-n}$ so that the roundoff error is eliminated for spatial grid discretization. The discrete time increment $dt$ and number of stages $s$ used in the super-time-stepping methods are also changed accordingly. For each time the spatial grid is refined, we divide $dt$ by 50 if $\|u(\cdot, t)\|_{L^\infty}<1000$, and divide $dt$ by 700 otherwise. The number of stages $s$ is decreased to ensure that the numerical stability is preserved, while we require $s\ge5$ to avoid the loss of precision caused by the excessively small $s$. For semilinear heat equation simulations, we start with $s=200$, and update $s$ by $s_{\text{new}}=\max(5, \lfloor0.45s\rfloor)$. For the axisymmetric surface diffusion equation, we start with $s=70$, and update $s$ by $s_{\text{new}}=s-5$ until $s=5$.

\section{Numerical Discretization}
\label{app:discretization}
The super-time-stepping method method is used to solve the diffusion equation using the method of lines discretization; more precisely, we solve the ODE $\frac{du}{dt}=\delta^2u$ where $\delta^2u$ is the central discretization of the second derivative of $u$ in space. Since we use adaptive mesh refinement in our method, the spatial grid is not uniform, and the corresponding central difference approximation to the 2-nd order spatial derivative is
\begin{equation}
    \Big(\frac{\partial^2 u}{\partial x^2}\Big)_i\approx\frac{u_{i+1}(x_i-x_{i-1})+u_{i-1}(x_{i+1}-x_i)-u_i(x_{i+1}-x_{i-1})}{\frac{1}{2}(x_{i+1}-x_{i-1})(x_{i+1}-x_{i})(x_{i}-x_{i-1})}.
\end{equation}
By elementary numerical analysis, the operator $\delta^2$ satisfies all the convergence requirement for $\mathbf{M}$ in the super-time-stepping method analysis. Then the non-linear ODE $u_t=u^p$ is solved using its analytic solution. Combining them using the Strang splitting, in each super timestep forward from $t_n$ to $t_n+dt$, we solve for
\begin{equation}
\begin{split}
    \frac{d\Tilde{u}}{dt}&=\Tilde{u}^p, \Tilde{u}(t_n)=u(t_n), t\in[t_n, t_n+\frac{1}{2}dt],\\
    \frac{du^*}{dt}&=\delta^2u^*, u^*(t_n)=\Tilde{u}(t_n+\frac{1}{2}dt), t\in[t_n, t_n+dt],\\
    \frac{du^\dagger}{dt}&=(u^\dagger)^p, u^\dagger(t_n+\frac{1}{2}dt)=\Tilde{u}(t_n+dt), t\in[t_n+\frac{1}{2}dt, t_n+dt],
\end{split}
\end{equation}
and finally we set $u(t_n+dt)=u^\dagger(t_n+dt)$. 

 In more detail, the right hand side is discretized as follows
\begin{equation}
\begin{aligned}
(\mathbf{M}r)_m^n&:=\frac{1}{2r_m^ndx}\Biggl[\frac{r_{m+1}^n}{\sqrt{1+(\delta r_{m+1}^n)^2}}\Biggl(\frac{\Tilde{\H}_{m+2}^n-\Tilde{\H}_m^n}{2dx}\Biggl)\\
&-\frac{r_{m-1}^n}{\sqrt{1+(\delta r_{m-1}^n)^2}}\Biggl(\frac{\Tilde{\H}_m^n-\Tilde{\H}_{m-2}^n}{2dx}\Biggl)\Biggl],\\
\end{aligned}
\end{equation}
where
\begin{equation}
\begin{aligned}
\Tilde{\H}_m^n&:=\frac{1}{r_m^n\sqrt{1+(\delta r_m^n)^2}}-\frac{\delta^2r_m^n}{(1+(\delta r_m^n)^2)^{3/2}}\\
\delta r_m^n&:=\frac{r_{m+1}^n-r_{m-1}^n}{2dx}\quad\delta^2 r_m^n:=\frac{r_{m+1}^n-2r_m^n+r_{m-1}^n}{dx^2}         
\end{aligned} 
\end{equation}
and we solve the ODE $\frac{dr}{dt}=\mathbf{M}r$ using the super-time-stepping method to get the approximated solution.

\section{Numerical Monotonicity Example}
\label{app:monotone_example}
To understand Definition~\ref{def:monotonicity} more clearly, consider solving the heat equation $u_t=u_{xx}$ using Euler's method; the numerical update is
$$u^{t+dt}_i=u^{t}_i+\frac{dt}{dx^2}(u^{t}_{i+1}-2u^{t}_i+u^{t}_{i-1}),$$
where $dt$ represents the timestep and $dx$ represents the spatial grid size. Then in this case, $j$ enumerates from $i-1$ to $i+1$ for a fixed grid index $i$; more explicitly,
$$\mathcal{H}(u_{i+1}^t, u_{i}^t, u_{i-1}^t)=u^{t}_i+\frac{dt}{dx^2}(u^{t}_{i+1}-2u^{t}_i+u^{t}_{i-1}).$$
Euler's method in this case is monotone, by applying the definition
$$\frac{\partial\mathcal{H}}{\partial u_{i+1}^t}=\frac{dt}{dx^2}\ge0;\quad\frac{\partial\mathcal{H}}{\partial u_{i-1}^t}=\frac{dt}{dx^2}\ge0;\quad\frac{\partial\mathcal{H}}{\partial u_i^t}=1-2\frac{dt}{dx^2}\ge0,$$
under the CFL constraint $\frac{dt}{dx^2}\le\frac{1}{2}$.

Throughout our proof, we focus on the problem~\ref{eqn:theoryheat}. In that case, for a $s$-stage stabilized Runge--Kutta method, $j$ enumerates from $j-s$ to $j+s$.

\section{Proof of the Alternating Summation Estimation Lemma~\ref{lemma:alterbinom}}
\label{app:sum_lemma}
The main idea is to write the finite sum as a generalized hypergeometric function in the form of the left hand side of Theorem~\ref{thm:clausen} up extracting a suitable positive $n, j$-dependent constant, so that the result immediately follows. The restriction on $x$ serves exactly as a condition that the series in Theorem~\ref{thm:clausen} are well-defined.
\begin{proof}[Proof of Lemma~\ref{lemma:alterbinom}]
By direct computation,
\begin{align*}
&\frac{j!j!(s-j)!}{(s+j)!}\sum_{k=j}^s\binom{s}{k}\binom{s+k}{k}x^k\binom{2k}{j+k}(-1)^{j+k}\\
=&\frac{j!j!(s-j)!}{(s+j)!}\sum_{k=0}^{s-j}\binom{s}{j+k}\binom{s+j+k}{j+k}(-x)^k\binom{2k+2j}{k+2j}\\
=&\frac{j!j!(s-j)!}{(s+j)!}\sum_{k=0}^{s-j}(-x)^k\frac{1}{(j+k)!(s-j-k)!}\frac{(s+j+k)!}{(j+k)!}\frac{(2k+2j)!}{(k+2j)!k!}\\
=&(s-j)\cdots(s-j-k+1)\sum_{k=0}^{s-j}(-x)^k\frac{(s+j+k)\cdots(s+j+1)}{(j+1)^2\cdots(j+k)^2}\frac{(2k+2j)\cdots(k+2j+1)}{k!}\\
=&\sum_{k=0}^{s-j}\frac{x^k}{k!}\frac{(s+j+k)\cdots(s+j+1)(j-n)\cdots(j+k-s-1)(2k+2j)\cdots(k+2j+1)}{(j+1)^2\cdots(j+k)^2}\\
=&\sum_{k=0}^{\infty}\frac{x^k}{k!}\frac{(j+s+1)_k(j-s)_k(2k+2j)\cdots(k+2j+1)}{(j+1)_k(j+1)\cdots(j+k)}.
\end{align*}
Now we show that for any $k\ge0$,
$$\frac{(2k+2j)\cdots(k+2j+1)}{(j+1)\cdots(j+k)}=4^k\frac{\Gamma(j+k+\frac{1}{2})}{\Gamma(j+\frac{1}{2})(2j+1)\cdots(2j+k)}.$$
We proceed using induction. Clearly the identity holds for $k=0, 1$. Assume that the above identity holds for some $k\in\N$; then by the inductive hypothesis and the induction relationship of the gamma function,
\begin{align*}
&4^{k+1}\frac{\Gamma(j+k+\frac{3}{2})}{\Gamma(j+\frac{1}{2})(2j+1)\cdots(2j+k+1)}\\
=&4^k\frac{\Gamma(j+k+\frac{1}{2})}{\Gamma(j+\frac{1}{2})(2j+1)\cdots(2j+k)}\frac{4(j+k+\frac{1}{2})}{(2j+k+1)}\\
=&\frac{(2k+2j)\cdots(k+2j+1)}{(j+1)\cdots(j+k)}\frac{2(2j+2k+1)}{2j+k+1}\\
=&\frac{2(k+j+1)(2k+2j+1)(2k+2j)\cdots(k+2j+2)}{(j+1)\cdots(j+k)(k+j+1)}\\
=&\frac{(2k+2j+2)\cdots(k+2j+2)}{(j+1)\cdots(j+k+1)}
\end{align*}
as desired. Therefore, by induction, the above identity is proved. Plugging back into the result of the beginning computation in this lemma, using definition~\ref{eqn:hypergeo_def},
$$\sum_{k=j}^s\binom{s}{k}\binom{s+k}{k}x^k\binom{2k}{j+k}(-1)^{j+k}=\frac{(s+j)!}{j!j!(s-j)!}\pFq{3}{2}{j+\frac{1}{2}, j+s+1, j-s}{j+1, 2j+1}{4x}.$$
Since $4x<1$, using Theorem~\ref{thm:clausen} with $a=\frac{j+s+1}{2}, b=\frac{j-s}{2}$, we know that
$$\pFq{3}{2}{j+\frac{1}{2}, j+s+1, j-s}{j+1, 2j+1}{4x}=\Big(\pFq{2}{1}{\frac{j+s+1}{2}, \frac{j-s}{2}}{j+1}{4x}\Big)^2\ge0$$
which proves the desired result for $x<\frac{1}{4}$.

By continuity of polynomials, taking $4x\to1$ in equation~\ref{eqn:combformula_rkl1} gives us the desired inequality when $x=\frac{1}{4}$. This concludes our proof.
\end{proof}

\begin{remark}
Notice that in the final step of the above lemma, Clausen's formula~\ref{thm:clausen} cannot be directly applied when $s-j$ is an odd number and at the same time $x\ge\frac{1}{4}$, since in that case $\frac{j-s}{2}\notin\Z$ and $\frac{j-s}{2}<0$, so that $\pFq{2}{1}{\frac{j+s+1}{2}, \frac{j-s}{2}}{j+1}{4x}$ is a \textit{series} that is not well-defined instead of a polynomial. This shows the necessity of the CFL condition in proving the numerical monotonicity.
\end{remark}

\section*{Acknowledgments}

The first author would like to thank for T. Amdeberhan and Professor Hjalmar Rosengren for providing an important observation in proving the RKL monotonicity combinatorial inequality~\cite{484999}.

\bibliographystyle{siamplain}
\bibliography{references}

\clearpage

\end{document}